\documentclass[11pt]{article}

\usepackage{amsmath,amssymb,amsthm,mathtools}
\usepackage{graphicx}
\usepackage{geometry}
\usepackage[autostyle=true]{csquotes}
\usepackage{bm}
\usepackage{enumitem}
\usepackage{booktabs}
\usepackage{algorithm}
\usepackage{algorithmic}

\usepackage{hyperref}
\hypersetup{
  colorlinks=true,
  linkcolor=blue,
  citecolor=blue,
  urlcolor=blue
}
\usepackage{cleveref}

\newcommand{\eps}{\varepsilon}
\newcommand{\dd}{\,\mathrm{d}}

\newcommand{\Vol}{\mathcal{V}}
\newcommand{\VQ}{\Vol_{Q}}                   
\newcommand{\dOm}{\dd \Omega}

\newcommand{\ErrV}{\mathrm{Err}_{\mathcal V}}
\newcommand{\ErrVint}{\ErrV^{\mathrm{int}}}
\newcommand{\ErrVbulk}{\ErrV^{\mathrm{bulk}}}

\newcommand{\Heav}{\Theta} 

\newtheorem{theorem}{Theorem}
\newtheorem{lemma}{Lemma}

\newtheorem{remark}{Remark}

\title{Third-Order Geometric-Volume Conservation in Cahn–Hilliard Models}
\author{Josef Musil$^\dagger$}
\date{\today}

\begin{document}
\maketitle

\begingroup
\renewcommand\thefootnote{$\dagger$}
\footnotetext{Institute of Thermomechanics, Czech Academy of Sciences, Dolej\v{s}kova 1402/5, 182 00 Prague 8, Czech Republic}
\endgroup

\begin{abstract}
Degenerate Cahn--Hilliard phase-field models provide a robust approximation of surface-diffusion--driven interface motion without explicit front tracking. In computations, however, the \emph{geometric} volume enclosed by the interface---e.g., the region where the order parameter $\phi\in[-1,+1]$ is positive---may drift at finite interface thickness, producing artificial shrinkage or growth even when the sharp-interface limit conserves volume.
Beyond interface-position-based mesh refinement (which reduces drift at increased computational cost), common remedies include enforcing volume conservation through global (nonlocal) constraints—typically implemented via Lagrange-multiplier corrections—and redesigning the dynamics to improve sharp-interface fidelity, ranging from doubly-degenerate surface-diffusion models (non-variational and energy-dissipative variants) to weighted-metric gradient-flow formulations with degenerate mobilities.
We revisit and extend the improved-conservation philosophy of Zhou et al., where one replaces classical mass conservation by the exact conservation of a designed monotone mapping $Q(\phi)$ that more accurately approximates a step function, yielding enhanced convergence of the geometric volume.
Building on this framework, we (i) carry out the matched-asymptotic analysis in the unscaled (physical) time formulation, (ii) derive a simplified, computable representation of the first-order inner correction to the interface profile and visualize its structure, and (iii) identify an \emph{integral-moment cancellation condition} that controls the leading geometric-volume defect. This mechanism becomes a practical design rule: we select regularization kernels within parameterized families---including exponential and Pad\'e-type---to reach effective higher-order behavior and satisfy the cancellation condition at moderate parameter values. As a result, the proposed kernels achieve \emph{formal third-order accuracy} in geometric-volume conservation with respect to interface thickness.
Finally, we describe an unconditional energy-dissipative numerical discretization that exactly preserves the discrete conserved quantity and improves robustness through a fully consistent update of the conserved mapping.
\end{abstract}

\paragraph{Keywords.}
Surface diffusion; degenerate Cahn--Hilliard; volume conservation; sharp-interface limit; inverse design; matched asymptotics.



\section{Introduction}\label{sec:intro}

Degenerate Cahn--Hilliard phase-field models \cite{cahn_1958,cahn_1961,elliott1996cahn} provide a robust diffuse-interface approximation of surface-diffusion--driven
interface motion without explicit front tracking. In the sharp-interface setting, surface diffusion evolves an embedded
interface $\Gamma(t)$ by a fourth-order geometric law (normal velocity proportional to the surface Laplacian of mean
curvature) while preserving the geometric volume enclosed by $\Gamma(t)$, a classical picture dating back to Mullins'
thermal-grooving theory and subsequent diffuse-interface formulations and sharp-interface analyses
\cite{mullins_1957,cahn_1996,lee_2016_sharp,salvalaglio_2021,bretin_2022}.

Phase-field formulations replace the sharp moving interface by an order parameter $\phi=\phi_\varepsilon\in[-1,1]$
whose transition layer has thickness $\mathcal{O}(\varepsilon)$. One typically has $\phi\approx +1$ in the
\enquote{inside} phase and $\phi\approx -1$ in the \enquote{outside} phase, with the diffuse interface located near
the zero level set. This diffuse description naturally accommodates topology changes and coupling to other physics,
but at finite interface thickness it raises a practical issue: the \emph{geometric} phase volumes
\eqref{eq:Omegaeps_pm}, associated with the thresholded field, may drift in time, producing artificial
shrinkage or growth even when the limiting sharp-interface dynamics conserves volume.
\begin{equation}\label{eq:Omegaeps_pm}
\Omega_{\varepsilon}^{+}(t):=\{x\in\Omega:\ \phi(x,t)>0\},\qquad
\Omega_{\varepsilon}^{-}(t):=\{x\in\Omega:\ \phi(x,t)<0\}.
\end{equation}
Denoting the diffuse interface by $\Gamma_{\varepsilon}(t):=\{x\in\Omega:\ \phi(x,t)=0\}$, we have
$\Omega=\Omega_\varepsilon^{+}(t)\cup\Gamma_\varepsilon(t)\cup\Omega_\varepsilon^{-}(t)$ and, in particular,
$|\Omega_\varepsilon^{+}(t)|+|\Omega_\varepsilon^{-}(t)|=|\Omega|$ (up to a set of measure zero).

\begin{figure}[h]
\centering
\includegraphics[width=\textwidth]{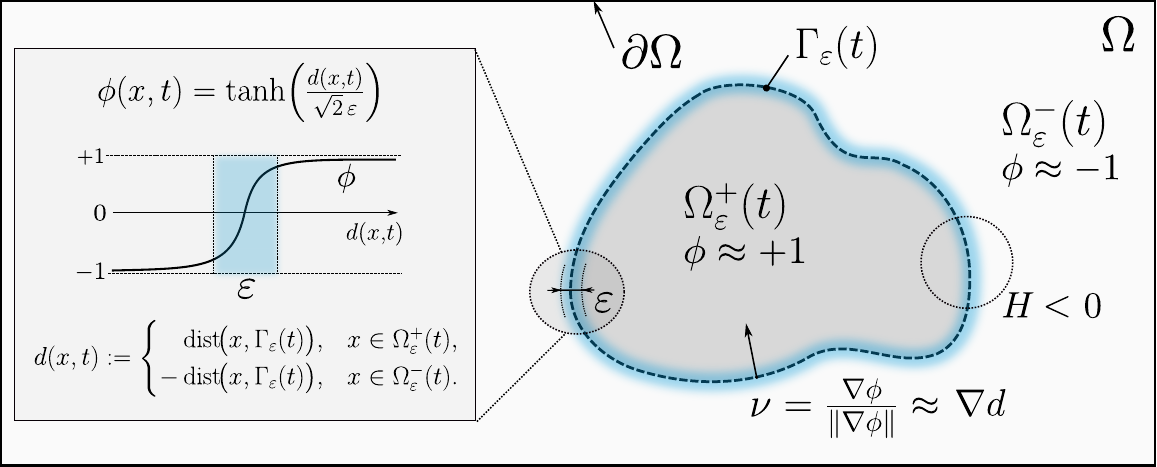}
\caption{Geometry and notation for the sharp interface limit.}
\label{fig:geometry}
\end{figure}

The origin of this drift is both modeling and numerical. Classical Cahn--Hilliard dynamics conserves a diffuse \enquote{mass}
(e.g.\ $\int_\Omega \phi\,\dOm$), which only approximates the sharp volume to first order in $\eps$ for typical inner
profiles; moreover, bulk diffusion and discretization errors can accumulate over long times when interfaces are not
sufficiently resolved, leading to spurious loss of small droplets and distorted coarsening dynamics
\cite{cahn_1996,lee_2016_sharp,musil_furst_2025}.
Interface-position-based mesh refinement can reduce these errors, but at increased computational cost.

Accordingly, the literature offers two broad remedy classes beyond refinement. A first class enforces volume conservation
through \emph{global (nonlocal) constraints}, typically implemented via Lagrange-multiplier corrections that depend on
domain integrals and adjust the chemical potential (or reaction term) so that a target volume proxy remains (nearly)
constant; such ideas are widely used across phase-field models, including conservative Allen--Cahn-type formulations and
related mass/volume-correction strategies in diffuse-interface fluid models
\cite{kim_2014_conservative_ac,shen_yang_wang_2013,wu_2021_review}.
A second class redesigns the \emph{local} model structure to improve sharp-interface fidelity while preserving an energy
dissipation principle. In the surface-diffusion setting, doubly-degenerate Cahn--Hilliard models introduce an additional
degeneracy to suppress bulk diffusion; a widely used non-variational version performs well computationally but lacks an
underlying energy, while variational (energy-dissipative) variants restore a gradient-flow structure
\cite{salvalaglio_2021}. Building on the variational viewpoint, weighted-metric gradient-flow formulations with
(two) degenerate mobilities have been proposed to achieve second-order accuracy in the sharp-interface approximation of
surface diffusion while retaining the classical Cahn--Hilliard energy \cite{bretin_2022}.

Recently, Zhou et al.\ introduced an \emph{anisotropic Cahn--Hilliard equation with improved conservation} (ACH--IC),
derived from Onsager's variational principle under a modified conservation law in which one replaces classical mass
conservation by the exact conservation of a designed monotone mapping $Q(\phi)$ \cite{zhou_2025_achic}.
The mapping $Q$ is chosen as a smooth surrogate of the sharp phase indicator: it is odd, satisfies $Q(\pm 1)=\pm 1$, and
is steeper near $\phi=0$ than $\phi$ itself. As a consequence, the conserved \enquote{$Q$-volume}
\begin{equation}\label{eq:VQ_def}
\VQ(t) \;:=\; \frac12\int_\Omega (1+Q(\phi(x,t)))\,\dOm
\end{equation}
provides a higher-fidelity diffuse approximation of the geometric volume, reducing spurious shrinkage at practical values
of $\eps$ \cite{zhou_2025_achic}. In Zhou's baseline family, $Q'$ is proportional to $(1-\phi^2)^k$, which improves the
sharp-indicator approximation as $k$ increases but can restrict the design space.

In this work we revisit and extend Zhou's improved-conservation framework in two directions motivated by analysis and by
numerical practice. First, we carry out the matched-asymptotic analysis in the \emph{unscaled (physical) time}
formulation (no $\eps$-dependent prefactor in $\partial_t\phi$), which is natural for multiphysics coupling (e.g.\
Cahn--Hilliard--Navier--Stokes) but shifts physical motion to the first correction order
\cite{magaletti_2013,eikelder_2024}. Second, we continue the inner expansion to extract the curvature-forced first profile
correction (denoted $\Phi_1$), derive simplified one-dimensional integral representations, and use these to
identify an \emph{integral-moment cancellation condition} controlling the leading geometric-volume defect.
This yields a practical inverse-design rule: by selecting kernels $Q$ within parameterized families---including
exponential and Pad\'e-type enrichments that \enquote{reach} high effective degeneracy without extreme polynomial powers---we
enforce the cancellation at moderate parameter values. Together with sufficient endpoint degeneracy of $Q'$, this leads to
\emph{formal third-order accuracy} in geometric-volume conservation with respect to interface thickness.

The remainder of the paper is organized as follows. Section~\ref{sec:model} introduces the CH--IC model class and
admissible kernels $Q$. Section~\ref{sec:asymp} summarizes the unscaled matched-asymptotic hierarchy, identifies the
first inner correction, and derives the geometric-volume error expansion with the corrected moment hierarchy. Section~\ref{sec:design} presents kernel families and an inverse-design workflow. Numerical method and
experiments are presented in Section~\ref{sec:numerics} (placeholders in this draft). Appendices collect detailed
asymptotic derivations and corrected moment formulas.

\section{Conservation-Improved Cahn--Hilliard (CH-IC) models and conserved $Q$-mappings}\label{sec:model}

We consider a conserved scalar order parameter $\phi=\phi_\varepsilon(x,t)\in[-1,1]$ on a bounded Lipschitz domain
$\Omega\subset\mathbb{R}^d$ ($d=2,3$; the analysis is also valid for $d=1$) over a time interval $t\in[0,T]$.
The phase field $\phi$ may represent, for instance, a rescaled concentration difference in a binary mixture, and we assume
$\phi(\cdot,t)\in H^1(\Omega)$ for each $t$.
We use the thresholded phase regions $\Omega_\varepsilon^\pm(t)$ and the mid-surface $\Gamma_\varepsilon(t)$ defined
in~\eqref{eq:Omegaeps_pm}. In the sharp-interface regime $\varepsilon\ll 1$, $\phi$ is close to $\pm 1$ in the bulk and
transitions across an $\mathcal{O}(\varepsilon)$-thin layer near $\Gamma_\varepsilon(t)$.

The thermodynamics are governed by the standard Ginzburg--Landau free energy for an isotropic binary mixture at constant temperature,
\begin{equation}\label{eq:energy}
E_\varepsilon[\phi]
=\int_\Omega \gamma\Bigl(\frac{1}{\varepsilon}W(\phi)+\frac{\varepsilon}{2}|\nabla\phi|^2\Bigr)\,\dOm,
\qquad
W(\phi)=\frac14(\phi^2-1)^2,
\end{equation}
where $\gamma>0$ is the (constant) surface-tension coefficient. For notational simplicity we set $\gamma=1$ by rescaling the energy.
The associated chemical potential is
\begin{equation}\label{eq:mu}
\mu=\frac{\delta E_\varepsilon}{\delta\phi}
=\frac{1}{\varepsilon}W'(\phi)-\varepsilon\Delta\phi.
\end{equation}

\medskip\noindent
\textbf{Improved conservation.}
Classical Cahn--Hilliard dynamics conserves the \enquote{mass} $\int_\Omega \phi\,\dOm$.
Following Zhou et al.\ \cite{zhou_2025_achic}, we instead impose the exact conservation of a designed monotone mapping $Q(\phi)$,
\begin{equation}\label{eq:Qcons}
\frac{d}{dt}\int_\Omega Q(\phi)\,\dOm
=\int_\Omega Q'(\phi)\,\partial_t\phi\,\dOm=0,
\end{equation}
where $Q:[-1,1]\to[-1,1]$ is odd, strictly increasing, and satisfies
\begin{equation}\label{eq:Qassum}
Q(\pm1)=\pm1,\qquad Q(-\phi)=-Q(\phi),\qquad Q'(\phi)>0\ \text{for }|\phi|<1.
\end{equation}

\medskip\noindent
\textbf{CH--IC dynamics.}
Using Onsager's variational principle with the constraint \eqref{eq:Qcons} (see \cite{zhou_2025_achic} and the discussion in
\cite{bretin_2022,salvalaglio_2021}), one arrives at the conservation-improved Cahn--Hilliard (CH--IC) evolution
\begin{align}\label{eq:chic}
\partial_t\phi
&= N(\phi)\,\nabla\cdot\!\Bigl(M(\phi)\,\nabla\bigl(N(\phi)\mu\bigr)\Bigr),
\qquad
N(\phi):=\frac{1}{Q'(\phi)}.
\end{align}
Here $\mu$ is given by \eqref{eq:mu} and $M(\phi)\ge 0$ is a (possibly degenerate) mobility. We parameterize it as
\begin{equation}\label{eq:mobility_deg}
M(\phi)=M_\ast\,(1-\phi^2)^{\ell},\qquad \ell\in\mathbb{N},
\end{equation}
where $\ell$ controls the strength of the endpoint degeneracy and $M_\ast>0$ is a constant mobility magnitude. In the
surface-diffusion regime, a common choice is the quartic degeneracy $\ell=2$, which suppresses bulk diffusion and yields
surface-diffusion motion in the sharp-interface limit \cite{lee_2016_sharp,salvalaglio_2021,bretin_2022}.

\begin{remark}[Mobility magnitude and effective time scale]\label{rem:mobility_scaling}
Throughout the analysis we treat $M_\ast=\mathcal{O}(1)$ in order to isolate the geometric structure of the sharp-interface
limit. Since the interfacial velocity law depends linearly on $M_\ast$, a scaling $M_\ast=\varepsilon^{a}\,\overline M$
with $\overline M=\mathcal{O}(1)$ simply rescales the interfacial time scale: if
\[
V_n=\varepsilon\,c_{\mathrm{SD}}\,M_\ast\,\Delta_\Gamma H+\mathcal{O}(\varepsilon^2)
\]
(where $c_{\mathrm{SD}}>0$ is the profile-dependent surface-diffusion constant; cf.\ Appendix~\ref{app:A}, and
$c_{\mathrm{SD}}=\tfrac49$ for $W(\phi)=\tfrac14(\phi^2-1)^2$ and $M(\phi)=(1-\phi^2)^2$), then
\[
V_n=\varepsilon^{a+1}c_{\mathrm{SD}}\,\overline M\,\Delta_\Gamma H+\mathcal{O}(\varepsilon^{a+2}),
\]
equivalently yielding an $\mathcal{O}(1)$ law in the slow time $\tau=\varepsilon^{a+1}t$. In hydrodynamic NSCH models one
additionally has an advective leading-order kinematics $V_n=u\cdot\nu+\cdots$, so the mobility scaling determines whether
the CH-driven relaxation enters as a higher-order correction or at leading order; see, e.g.,
\cite{magaletti_2013,jacqmin1999calculation,demont2023numerical,abels2014sharp}.
\end{remark}
\noindent Henceforth we set $M_\ast=1$ without loss of generality (this amounts to a constant rescaling of time), and we write
$M(\phi)=(1-\phi^2)^{\ell}$.

The conserved diffuse \enquote{volume} proxy is
\begin{equation}
\mathcal{V}_Q(t):=\frac12\int_\Omega \bigl(1+Q(\phi(x,t))\bigr)\,\dOm,
\end{equation}
which is exactly constant under \eqref{eq:chic} with no-flux boundary conditions. The goal is to choose $Q$ so that
$\mathcal{V}_Q(t)$ approximates the \emph{geometric} volume $|\Omega^+(t)|$ with higher asymptotic accuracy.

\medskip
\noindent
\paragraph{Zhou's polynomial family.}
The baseline family introduced by Zhou et al. \cite{zhou_2025_achic} is defined by
\begin{equation}\label{eq:Qk_def}
Q_k'(\phi)=\frac{(1-\phi^2)^k}{B_k},
\qquad
B_k:=\int_0^1 (1-s^2)^k\,ds,
\qquad
Q_k(\phi):=\int_0^{\phi} Q_k'(s)\,ds,
\end{equation}
where $k\ge 0$ controls the endpoint degeneracy of $Q_k'$ at $\phi=\pm1$ and the normalization $B_k$ ensures $Q_k(1)=1$.
The case $k=0$ recovers the classical mass mapping $Q_0(\phi)=\phi$.
A particularly important instance is $k=1$, for which
$Q_1'(\phi)=\tfrac32(1-\phi^2)$ and $N(\phi)=\tfrac{2}{3}(1-\phi^2)^{-1}$.
In this case, \eqref{eq:chic} coincides with the weighted-metric $H^{-1}$ gradient flow underlying the
second-order variational surface-diffusion model of Bretin et al.~\cite{bretin_2022}, where the metric weight appears both
in the time derivative and in the flux. This links Zhou's improved-conservation mapping to weighted-metric surface-diffusion
formulations~\cite{bretin_2022,salvalaglio_2021}.
A distinctive feature of the NMN ($k=1$) case is \emph{first-order structural stability} of the inner profile under curvature:
the curvature-induced correction at order $\varepsilon$ vanishes identically, $\Phi_1\equiv 0$
(Section~\ref{sec:asymp}).

\begin{figure}[h]
\centering
\includegraphics[width=\textwidth]{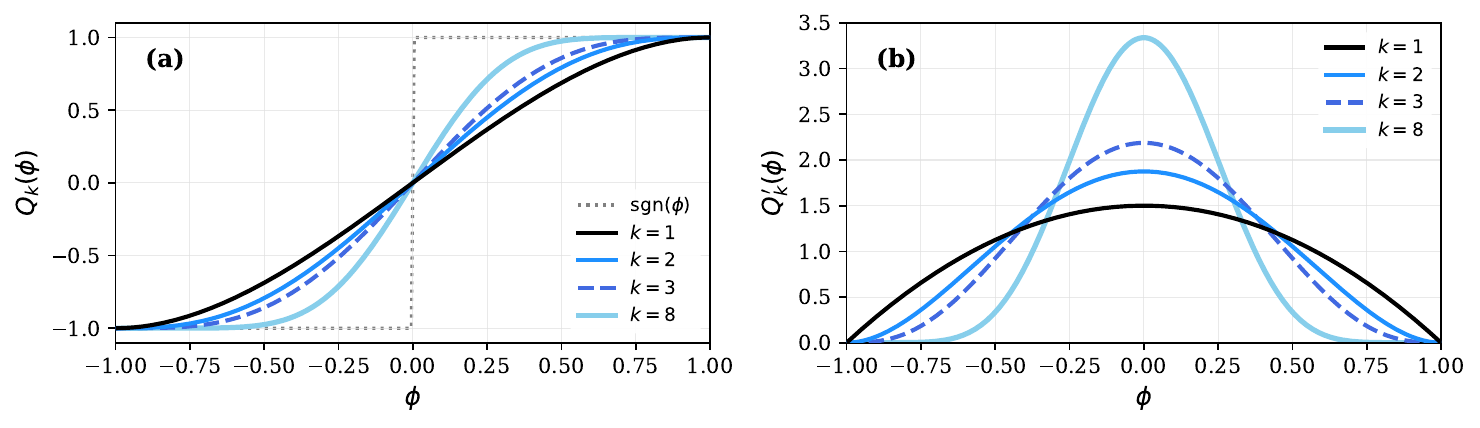}
\caption{Conservation mapping $Q_k(\phi)$ and its derivative $Q_k'(\phi)$ for Zhou's polynomial family~\eqref{eq:Qk_def}
with degeneracy orders $k = 1,2,3,8$.
(a)~The mapping $Q_k$ increasingly approximates the sharp sign function $\operatorname{sgn}(\phi)$ (dotted grey) as $k$ grows.
(b)~The derivative $Q_k'$ concentrates near $\phi = 0$ and vanishes at $\phi = \pm 1$ to order~$k$, controlling the endpoint
degeneracy. The case $k = 1$ corresponds to the NMN model of Bretin et al.~\cite{bretin_2022}.}
\label{fig:Q_profiles}
\end{figure}

\medskip\noindent
The polynomial family \eqref{eq:Qk_def} already improves the convergence of $\mathcal{V}_Q(t)$ to $|\Omega^+(t)|$ as $k$ increases.
However, the main analytical point of this paper is that the leading geometric-volume error depends on \emph{two} moments and can be
eliminated by a moment-balance condition; we summarize the mechanism and the resulting design rule next. We then exploit the freedom in $Q'$ to design kernels that enforce moment balance within low-dimensional families (Section~\ref{sec:design}).

\section{Unscaled asymptotics, geometric-volume error, and moment balancing}
\label{sec:asymp}

We summarize the unscaled matched-asymptotic structure and isolate the mechanism responsible for geometric-volume drift at finite
interface thickness. The key new ingredient relative to standard degenerate Cahn--Hilliard asymptotics is the
$\mathcal{O}(\varepsilon)$ curvature-induced inner correction triggered by the improved-conservation mapping $Q$.
Detailed derivations are given in Appendices~\ref{app:A}--\ref{app:B}.

\medskip\noindent
\textbf{Inner coordinate and leading profile.}
Let $d(x,t)$ denote the signed distance to $\Gamma(t)$ (positive in $\Omega^+(t)$), and introduce the stretched normal coordinate
$z=d(x,t)/\varepsilon$.
The leading inner profile is the heteroclinic $\sigma$ solving
\begin{equation}\label{eq:sigma}
\sigma''-W'(\sigma)=0,\qquad \sigma(\pm\infty)=\pm1,\qquad \sigma(0)=0,
\end{equation}
so that, for the quartic potential~\eqref{eq:energy},
\begin{equation}\label{eq:sigma_sol}
\sigma(z)=\tanh\!\left(\frac{z}{\sqrt{2}}\right).
\end{equation}
In the unscaled (physical) time formulation~\eqref{eq:chic}, the interface motion appears one order later than in the
classical $\varepsilon$-rescaled time setting. More precisely, the normal velocity admits the expansion
\begin{equation}\label{eq:Vlimit}
V_n \;=\; \varepsilon\,V_1 \;+\;\mathcal{O}(\varepsilon^2),
\end{equation}
and the leading coefficient $V_1$ is governed by surface diffusion,
\begin{equation}\label{eq:V1_SD}
V_1 \;=\; -\,\nabla_\Gamma\!\cdot\!\bigl(\mathcal{M}_\Gamma \nabla_\Gamma H\bigr),
\end{equation}
where the effective surface mobility $\mathcal{M}_\Gamma$ depends on the choices of $M(\phi)$ and $Q$.
In the isotropic case with the standard quartic potential and the baseline normalization used in this paper,
\eqref{eq:V1_SD} reduces to the constant-mobility law $V_1=\tfrac{4}{9}\,\Delta_\Gamma H$.
A matched-asymptotic derivation and the expression for $\mathcal{M}_\Gamma$ are summarized in Appendix~\ref{app:A};
see also \cite{lee_2016_sharp,salvalaglio_2021,bretin_2022,zhou_2025_achic}.

\medskip\noindent
\textbf{First curvature correction $\Phi_1$ and a reduced formula.}
We expand
\[
\phi(x,t)=\sigma(z)+\varepsilon\,H(x,t)\,\Phi_1(z)+\mathcal{O}(\varepsilon^2).
\]
At order $\varepsilon$, the correction $\Phi_1$ solves
\begin{equation}\label{eq:Phi1_eq}
\mathcal{L}\Phi_1
=
\frac{\sqrt{2}}{3}\,Q'(\sigma(z))-\sigma'(z),
\qquad
\mathcal{L}:=\partial_{zz}-W''(\sigma(z)),
\end{equation}
with boundedness as $z\to\pm\infty$ and gauge $\Phi_1(0)=0$.
The operator $\mathcal{L}$ is self-adjoint with $\ker\mathcal{L}=\mathrm{span}\{\sigma'\}$, and solvability follows from the Fredholm
alternative (Appendix~\ref{app:B}).

Zhou et al.\ represent the solution by a two-fold integral (variation-of-constants) formula:
\begin{equation}\label{eq:Phi1_Zhou}
\Phi_1(z)=
\sigma'(z)\int_{0}^{z}\frac{1}{\sigma'(\eta)^2}
\left(
\int_{0}^{\eta}
\Bigl(\frac{\sqrt{2}}{3}\,Q'(\sigma(\xi))-\sigma'(\xi)\Bigr)\sigma'(\xi)\,d\xi
\right)d\eta,
\end{equation}
which is correct but not ideal for analysis or kernel design.
A main simplification of this work is to reduce \eqref{eq:Phi1_Zhou} to a single quadrature in the phase variable
$u=\sigma(z)\in(-1,1)$.
Introduce the NMN reference kernel
\[
Q_1(u):=\frac{3}{2}\left(u-\frac{u^3}{3}\right),
\qquad\text{(corresponding to $k=1$, $S\equiv1$)}
\]
and define
\begin{equation}\label{eq:Fdef}
F(u):=\int_{0}^{u}\Bigl(\frac{\sqrt{2}}{3}Q'(w)-\frac{1}{\sqrt{2}}(1-w^2)\Bigr)\,dw
=\frac{\sqrt{2}}{3}\Bigl(Q(u)-Q_1(u)\Bigr).
\end{equation}
Then (see Appendix~\ref{app:B})
\begin{equation}\label{eq:Phi1_simplified}
\Phi_1(z)=\sigma'(z)\,2\sqrt{2}\int_{0}^{\sigma(z)}\frac{F(u)}{(1-u^2)^3}\,du
=\frac{4}{3}\,\sigma'(z)\int_{0}^{\sigma(z)}\frac{Q(u)-Q_1(u)}{(1-u^2)^3}\,du.
\end{equation}
This representation makes the structure transparent: $\Phi_1\equiv0$ if and only if $Q\equiv Q_1$, i.e.\ the NMN case is
\emph{structurally stable} at order~$\varepsilon$.

\medskip\noindent
\textbf{Geometric-volume defect and moment hierarchy.}
Although $\int_\Omega Q(\phi)\,\dOm$ is conserved exactly, the \emph{geometric} volume $|\Omega^+(t)|$ can drift at finite $\varepsilon$.
We therefore compare $|\Omega^+(t)|$ with the conserved diffuse $Q$-volume proxy $\VQ(t)$ defined in~\eqref{eq:VQ_def}, and introduce the
geometric-volume error
\begin{equation}\label{eq:EVdefinition}
\ErrV(t):=\VQ(t)-|\Omega^+(t)|
=\frac12\int_\Omega \bigl(1+Q(\phi(x,t))\bigr)\,\dOm-\int_\Omega \Heav(d(x,t))\,\dOm,
\end{equation}
where $\Heav$ denotes the Heaviside step function (equivalently, $\Heav(d(\cdot,t))$ is the indicator of $\Omega^+(t)$ up to the interface).
A tubular-neighborhood expansion yields the decomposition
\begin{equation}\label{eq:EVexpansion}
\ErrV(t)
=
\underbrace{\varepsilon\int_{\Gamma(t)} C_0[Q]\,dA
+\frac{\varepsilon^2}{2}\int_{\Gamma(t)} H\,\mathcal{C}_1[Q]\,dA}_{=:\ \ErrVint(t)}
\;+\;\ErrVbulk(t)
\;+\;\mathcal{O}(\varepsilon^3).
\end{equation}
For odd $Q$ satisfying \eqref{eq:Qassum}, the $\mathcal{O}(\varepsilon)$ coefficient
\begin{equation}\label{eq:C0_def}
C_0[Q]:=\int_{-\infty}^{\infty}\Bigl(\tfrac12\bigl(1+Q(\sigma(z))\bigr)-\Heav(z)\Bigr)\,dz
\end{equation}
vanishes, hence the leading interfacial contribution is $\mathcal{O}(\varepsilon^2)$.
Moreover, for kernels of the form \eqref{eq:Qprime_kS} with degeneracy order~$k$,
$\ErrVbulk(t)=\mathcal{O}(\varepsilon^{k+1})$ (Appendix~\ref{app:B}).

The $\mathcal{O}(\varepsilon^2)$ interfacial defect is governed by the combined moment
\begin{equation}\label{eq:C1def}
\mathcal{C}_1[Q]:=\mathcal{M}_1[Q]+\mathcal{J}_1[Q].
\end{equation}
The geometric moment depends only on $Q'$,
\begin{equation}\label{eq:M1}
\mathcal{M}_1[Q]=-2\int_0^1 Q'(u)\,\operatorname{arctanh}^2(u)\,du,
\end{equation}
while the dynamic moment captures the contribution from the curvature correction $\Phi_1$:
\begin{equation}\label{eq:J1}
\mathcal{J}_1[Q]=\frac{8}{3}\int_0^1 \frac{\bigl(Q(u)-Q_1(u)\bigr)\,\bigl(1-Q(u)\bigr)}{(1-u^2)^3}\,du,
\qquad
Q_1(u)=\frac{3}{2}\left(u-\frac{u^3}{3}\right).
\end{equation}
Both formulas follow from \eqref{eq:Phi1_simplified} by a change of variables and Fubini's theorem; see Appendix~\ref{app:B}.
The sign convention in \eqref{eq:Phi1_eq}--\eqref{eq:J1} is consistent with Zhou \cite{zhou_2025_achic} and yields $\mathcal{J}_1>0$ in the
moment-balanced regime.

\begin{lemma}\label{lem:M1negative}
For any strictly increasing, odd kernel $Q$ satisfying \eqref{eq:Qassum}, one has $\mathcal{M}_1[Q]<0$.
\end{lemma}

The negativity of $\mathcal{M}_1$ reflects the fact that $Q(\sigma(z))$ always \enquote{sags} below the ideal step function across the diffuse layer.
Hence cancellation of the $\mathcal{O}(\varepsilon^2)$ volume error requires a positive contribution from $\mathcal{J}_1$.

\medskip\noindent
\textbf{Design rule (moment balance).}
Third-order geometric-volume accuracy is obtained provided:
\begin{equation}\label{eq:design_rule}
\begin{aligned}
\text{(i)}\;& Q'(\phi)\ \text{vanishes at least to second order at}\ \phi=\pm1,\\
\text{(ii)}\;& \mathcal{C}_1[Q]=\mathcal{M}_1[Q]+\mathcal{J}_1[Q]=0.
\end{aligned}
\end{equation}
Under (i)--(ii), $\ErrVbulk(t)=\mathcal{O}(\varepsilon^3)$ and the leading interfacial contribution cancels, so that
$\ErrV(t)=\mathcal{O}(\varepsilon^3)$.

\section{Kernel families and inverse design}\label{sec:design}

The moment formulas in Section~\ref{sec:asymp} enable an inverse-design viewpoint: within a low-dimensional kernel family, we tune
parameters so that the combined moment $\mathcal{C}_1[Q]=\mathcal{M}_1[Q]+\mathcal{J}_1[Q]$ vanishes, while enforcing sufficient endpoint
vanishing of $Q'$ to retain third-order bulk behavior.

We parameterize kernels through a nonnegative, even shaping function $S(\phi;\theta)$ and a baseline exponent $k\ge1$:
\begin{equation}\label{eq:Qprime_kS}
Q'(\phi)=\frac{(1-\phi^2)^k\,S(\phi;\theta)}{B_{k,S}(\theta)},
\qquad
B_{k,S}(\theta)=\int_0^1 (1-s^2)^k\,S(s;\theta)\,ds,
\qquad
Q(\phi)=\int_{0}^{\phi}Q'(s)\,ds.
\end{equation}
This construction guarantees $Q(\pm1)=\pm1$, $Q$ odd, and $Q'>0$ on $(-1,1)$.
The factor $(1-\phi^2)^k$ enforces baseline vanishing of order~$k$ at $\phi=\pm1$.
If, in addition, $S(\pm1;\theta)=0$ with $S(\phi)\sim c(1-\phi^2)$ near the endpoints, then $(1-\phi^2)^k S(\phi)$ vanishes like
$(1-\phi^2)^{k+1}$; that is, shaping contributes one extra endpoint power. This is useful for $k=1$: endpoint-vanishing shaping yields effective
bulk degeneracy two while preserving the analytical advantages of the NMN baseline.

For any fixed $(k,S)$, moment balance becomes a one-dimensional root problem,
\begin{equation}\label{eq:C1_root}
\mathcal{C}_1(\theta):=\mathcal{M}_1[Q_\theta]+\mathcal{J}_1[Q_\theta]=0,
\end{equation}
with $\mathcal{M}_1$ and $\mathcal{J}_1$ evaluated by quadratures using \eqref{eq:M1}--\eqref{eq:J1}.
Figure~\ref{fig:kernel_families}(d) illustrates that $\mathcal{C}_1(\theta)$ varies smoothly and crosses zero robustly for the shaped families, so
$\theta^\star$ can be found by standard bracketing.

\subsection{Kernel families}\label{ssec:families}

\medskip\noindent
\textbf{Zhou's polynomial family ($S\equiv1$).}
Setting $S\equiv1$ recovers Zhou's baseline kernels $Q'_k(\phi)\propto(1-\phi^2)^k$.
The combined moment $\mathcal{C}_1$ decreases with~$k$ and crosses zero near $k=8$.
The case $k=1$ (NMN) is special: $\Phi_1\equiv0$ identically, but $\mathcal{C}_1\approx-0.645$.

\medskip\noindent
\textbf{Exponential shaping.}
For $k\ge2$ we use $S(\phi)=\exp(\beta_2\phi^2)$ with $\beta_2<0$.
For endpoint-vanishing designs with $k=1$, we use
\begin{equation}\label{eq:S_exp_EV}
S(\phi)=\exp(\beta_2\phi^2)-\exp(\beta_2),\qquad \beta_2<0,
\end{equation}
which is even, nonnegative on $[-1,1]$, and satisfies $S(\pm1)=0$.

\medskip\noindent
\textbf{Simple rational (endpoint-vanishing).}
\begin{equation}\label{eq:S_simple_EV}
S(\phi)=\frac{1}{1+q\phi^2}-\frac{1}{1+q},
\qquad q>0.
\end{equation}

\medskip\noindent
\textbf{Pad\'e-type (endpoint-vanishing).}
\begin{equation}\label{eq:S_pade_EV}
S(\phi)=\frac{1+p\,\phi^4}{1+q\,\phi^2}-\frac{1+p}{1+q},
\qquad q>0.
\end{equation}

\smallskip
Table~\ref{tab:kernel_summary} collects the tuned parameters and moment values.
Figure~\ref{fig:kernel_families} visualizes the mappings, kernels, and inner corrections.

\begin{figure}[h]
\centering
\includegraphics[width=\textwidth]{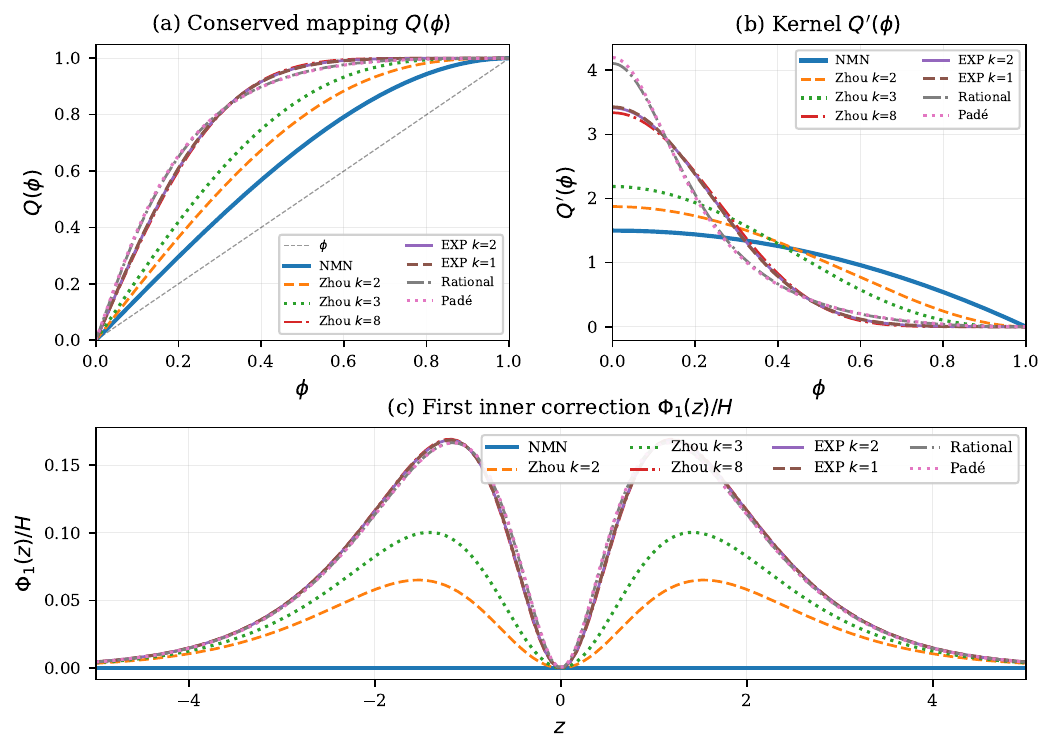}
\caption{Comparison of kernel families.
(a)~Conserved mapping $Q(\phi)$.
(b)~Kernel $Q'(\phi)$: moment-balanced designs concentrate near $\phi=0$.
(c)~First inner correction $\Phi_1(z)/H$: NMN (solid) has $\Phi_1\equiv 0$; others develop curvature-induced distortion.
(d)~Combined moment $\mathcal{C}_1$: five balanced designs achieve $\mathcal{C}_1\approx0$.}
\label{fig:kernel_families}
\end{figure}

\begin{table}[h!]
\centering
\caption{Summary of kernel families. NMN uniquely has $\Phi_1\equiv 0$.
Endpoint-vanishing $k{=}1$ families have effective degeneracy~2.}
\label{tab:kernel_summary}
\smallskip
\small
\begin{tabular}{@{}r l c l l r r r c@{}}
\toprule
 & Model & $k$ & Shaping $S(\phi)$ & Parameters & $\mathcal{M}_1$ & $\mathcal{J}_1$ & $\mathcal{C}_1$ & $\bigl\|\Phi_1/H\bigr\|_\infty$ \\
\midrule
1 & NMN          & 1 & $1$  & --- & $-0.645$ & $0$       & $-0.645$    & $0$     \\
2 & Zhou $k{=}2$ & 2 & $1$  & --- & $-0.395$ & $0.090$  & $-0.305$    & $0.065$ \\
3 & Zhou $k{=}3$ & 3 & $1$  & --- & $-0.284$ & $0.113$  & $-0.171$    & $0.100$ \\
4 & Zhou $k{=}8$ & 8 & $1$  & --- & $-0.118$ & $0.118$  & $\approx 0$ & $0.169$ \\
\addlinespace
5 & EXP $k{=}2$  & 2 & $e^{\beta_2\phi^2}$                         & $\beta_2{=}{-}6.95$             & $-0.121$ & $0.121$ & $\approx 0$ & $0.169$ \\
6 & EXP $k{=}1$  & 1 & $e^{\beta_2\phi^2}{-}e^{\beta_2}$           & $\beta_2{=}{-}8.12$             & $-0.121$ & $0.121$ & $\approx 0$ & $0.169$ \\
7 & Rational     & 1 & $\frac{1}{1+q\phi^2}{-}\frac{1}{1+q}$       & $q{=}20.9$                       & $-0.139$ & $0.139$ & $\approx 0$ & $0.167$ \\
8 & Pad\'e       & 1 & $\frac{1+p\phi^4}{1+q\phi^2}{-}\frac{1+p}{1+q}$ & $p{=}{-}0.30,\, q{=}23.4$     & $-0.140$ & $0.140$ & $\approx 0$ & $0.167$ \\
\bottomrule
\end{tabular}
\end{table}

\medskip\noindent
\emph{Practical note for numerics.}
All normalizations and moments required for tuning are one-dimensional integrals and are evaluated once per kernel.
In the solver, unshaped cases are used in closed form, while shaped kernels are tabulated on $[-1,1]$ and interpolated at runtime
(Section~\ref{sec:numerics}).

\section{Numerical methods}\label{sec:numerics}

We discretize the CH--IC system by a cell-centered finite-volume method on general polyhedral meshes
in a foam-extend / OpenFOAM-style field-operator framework; cf.\ \cite{weller_1998}.
Volume integrals are approximated by cell averages, and all diffusion operators are written in conservative
(face-flux) form using Gauss' theorem with standard non-orthogonal corrections.
Throughout, we impose no-flux boundary conditions for all diffusive fluxes, i.e.,
$\mathbf{n}\!\cdot\!\nabla\phi=0$ and $\mathbf{n}\!\cdot\!\nabla\psi=0$ on $\partial\Omega$.

\paragraph{Mixed CH--IC form and rescaled chemical potential.}
We solve the conservation-improved Cahn--Hilliard dynamics in the mixed form
\begin{subequations}\label{eq:chic_mixed}
\begin{align}
\partial_t Q(\phi) &= \nabla\cdot\!\Big(M(\phi)\,\nabla \psi\Big),
\label{eq:chic_mixed_Q}
\\
Q'(\phi)\,\psi &= \mu
= \frac{1}{\varepsilon}W'(\phi)\;-\;\varepsilon\,\Delta\phi\;+\;R(\phi,\nabla\phi).
\label{eq:chic_mixed_psi}
\end{align}
\end{subequations}
Here $Q$ is the designed monotone phase-indicator surrogate, and
$\psi := \mu/Q'(\phi)$ is the \emph{rescaled} chemical potential.
The mixed form is advantageous in finite volumes: it avoids repeatedly differentiating $Q$ inside fluxes and
leads naturally to a robust block-coupled solve for $(\phi,\psi)$.
In this work we focus on the $Q$-conservative CH--IC family and therefore set $R(\phi,\nabla\phi)\equiv 0$ in what follows.
(The implementation also supports additional variational terms of the V-model type from \cite{musil_furst_2025_enhanced};
they are omitted here.)

\paragraph{Exact discrete conservation update.}
Let $\delta t=t^{n+1}-t^n$. The defining feature of the scheme is that the conserved mapping is advanced by
its \emph{true discrete increment}
\begin{equation}\label{eq:Q_true_increment}
\frac{Q(\phi^{n+1})-Q(\phi^{n})}{\delta t},
\end{equation}
rather than by the surrogate $Q'(\phi)\,\partial_t\phi$.
Because the spatial discretization is conservative, the global discrete invariant
\[
\sum_{i=1}^{N} Q(\phi_i^{n+1})\,|\Omega_i|
=
\sum_{i=1}^{N} Q(\phi_i^{n})\,|\Omega_i|
\]
is preserved up to linear-solver tolerance once the nonlinear loop converges,
where $|\Omega_i|$ denotes the volume of the $i$-th finite-volume cell.

\paragraph{Practical evaluation.}
For the Zhou polynomial family $Q_k$, $Q(\phi)$ admits an exact polynomial representation
$Q(\phi)=\bar Q(\phi)\,\phi$ where $\bar Q$ is an \emph{even} polynomial of degree $2k$.
For instance, $\bar Q_1(\phi)=(3-\phi^2)/2$ for NMN and
$\bar Q_2(\phi)= \tfrac{15}{8} - \tfrac{5}{4}\phi^2 + \tfrac{3}{8}\phi^4$ for Zhou $k=2$.
These are evaluated using exact polynomial arithmetic (no numerical integration at runtime).
For shaped kernels, $Q(\phi)$ is precomputed on a uniform grid of $N_{\mathrm{tab}}=256$ nodes over $[0,1]$,
exploiting odd symmetry to reconstruct $Q$ on $[-1,1]$. Each entry is obtained by 16-point Gauss--Legendre quadrature
($\sim\!10^{-14}$ relative precision); runtime uses cubic Hermite interpolation
(pointwise errors $<10^{-10}$).
We then evaluate $\bar Q(\phi)=Q(\phi)/\phi$ for $|\phi|>\phi_{\mathrm{tol}}$, and use the l'Hospital limit
$\bar Q(0)=Q'(0)$ for $|\phi|\le \phi_{\mathrm{tol}}$ (we take $\phi_{\mathrm{tol}}\sim10^{-10}$).
This avoids any stabilization offset in \eqref{eq:chic_mixed_Q} even for strongly degenerate kernels.
Regularization (if needed) is applied only in \eqref{eq:chic_mixed_psi} through a localized floor on $Q'$ near the pure phases; see below.

\paragraph{Nonlinear solve, Eyre-type stabilization, and localized regularization.}
Within each time step we use a Picard (fixed-point) iteration $k=0,1,2,\dots$ (here $k$ is an iteration index, distinct from the kernel index in $Q_k$).
Coefficient fields such as $M(\phi)$, $\bar Q(\phi)$, and $Q'(\phi)$ are evaluated at the current iterate $\phi^k$.
To obtain robustness for stiff interfacial dynamics, we employ an Eyre-type stabilization \cite{eyre_1998} by writing
\[
W(\phi)=\underbrace{\Big(W(\phi)+\frac{\beta}{2}\phi^2\Big)}_{=:W_c(\phi)}
\;-\;
\underbrace{\Big(\frac{\beta}{2}\phi^2\Big)}_{=:W_e(\phi)},
\]
so that both $W_c$ and $W_e$ are convex on $\mathbb{R}$ (for the quartic potential, $\beta\ge 1$ ensures
$W_c''(\phi)=W''(\phi)+\beta\ge 0$ for all $\phi$).
Accordingly,
\[
W'(\phi)=W_c'(\phi)-W_e'(\phi),\qquad
W_c'(\phi)=W'(\phi)+\beta\phi,\qquad W_e'(\phi)=\beta\phi .
\]
To obtain a robust semi-implicit update within the Picard loop, we use the stabilized linearization
\begin{equation}\label{eq:Wprime_beta_big}
W'(\phi^{n+1})
=
\underbrace{\bigl(W''(\phi^k)+\beta\bigr)\,\phi^{n+1}}_{\text{implicit (stabilized) part}}
\;-\;
\underbrace{\Big(\bigl(W''(\phi^k)+\beta\bigr)\phi^k - W'(\phi^k)\Big)}_{\text{explicit remainder (previous iterate)}},
\end{equation}
obtained by adding and subtracting $(W''(\phi^k)+\beta)\phi^k$.
The coefficient $W''(\phi^k)+\beta$ is nonnegative by construction, ensuring the implicit contribution is coercive.
In all computations reported here we fix $\beta=1.02$.

In the rescaled-potential equation \eqref{eq:chic_mixed_psi}, $Q'(\phi)$ degenerates at $\phi=\pm1$ for the
designed kernels. Numerically, this is handled by a \emph{localized floor} that activates only where $Q'$
falls below a prescribed threshold:
\[
Q'(\phi)\ \mapsto\ Q'_\alpha(\phi):=\max\!\big(Q'(\phi),\alpha_{Q'}\big),
\]
with $\alpha_{Q'}\ll1$. This improves conditioning while leaving the interfacial region effectively unchanged.
(We test sensitivity to $\alpha_{Q'}$ in the results section.)

\paragraph{Block-coupled linear system in explicit operator notation.}
At each Picard iterate $k$, we solve a coupled linear system for $(\phi^{n+1},\psi^{n+1})$.
Define the diffusion operator acting on $\psi$,
\[
\mathcal{D}_k(\,\cdot\,):= -\,\nabla\cdot\!\Big(M(\phi^k)\,\nabla(\,\cdot\,)\Big),
\]
and the linearized \enquote{Allen--Cahn} operator acting on $\phi$,
\[
\mathcal{K}_k(\,\cdot\,):= -\varepsilon\,\Delta(\,\cdot\,)
+ \frac{1}{\varepsilon}\,\bigl(W''(\phi^k)+\beta\bigr)\,(\,\cdot\,).
\]
Using an implicit Euler step for the \emph{exact} $Q$-increment and freezing nonlinear coefficients at $\phi^k$,
the Picard step can be written as the block system
\begin{equation}\label{eq:block_system_operator}
\begin{bmatrix}
\frac{\bar Q(\phi^k)}{\delta t}\,I & \mathcal{D}_k\\[3pt]
\mathcal{K}_k & -\,Q'_\alpha(\phi^k)\,I
\end{bmatrix}
\begin{bmatrix}
\phi^{n+1}\\[2pt]
\psi^{n+1}
\end{bmatrix}
=
\begin{bmatrix}
\frac{1}{\delta t}\,Q(\phi^{n})\\[4pt]
b_\psi(\phi^k)
\end{bmatrix},
\end{equation}
where $\bar Q(\phi):=Q(\phi)/\phi$ with $\bar Q(0)=Q'(0)$, and the right-hand side of the potential equation is
\begin{equation}\label{eq:bpsi_def}
b_\psi(\phi^k)
=
\frac{1}{\varepsilon}\Big(
W'(\phi^k) - \bigl(W''(\phi^k)+\beta\bigr)\phi^k
\Big).
\end{equation}

\begin{remark}[Advection and coupled transport]
If advection is present, it is incorporated in \eqref{eq:chic_mixed_Q} in standard finite-volume form as a conservative
face-flux divergence, optionally split into implicit/explicit parts.
We refer to \cite{musil_furst_2025_enhanced} for the convection--diffusion coupling and related implementation details.
\end{remark}

\begin{algorithm}[h]
\caption{Block-coupled CH--IC solver (per time step)}
\begin{algorithmic}[1]\label{alg:coupled_chic}
\STATE \textbf{Given:} $\phi^n$, $\psi^n$ (or $\mu^n$), $\delta t$, tolerance $\mathrm{tol}$.
\FOR{each time step $t^n\to t^{n+1}$}
  \STATE Set $k=0$, $\phi^k\gets \phi^n$, $\psi^k\gets \psi^n$.
  \REPEAT
    \STATE Update coefficients: $M(\phi^k)$, $Q(\phi^k)$, $\bar Q(\phi^k)$, $Q'_\alpha(\phi^k)$.
    \STATE Assemble the block system \eqref{eq:block_system_operator}--\eqref{eq:bpsi_def}.
    \STATE Solve for $(\phi^{n+1},\psi^{n+1})$ using GMRES + ILUC0 on the block matrix.
    \STATE (Optional) enforce bounds: $\phi^{n+1}\gets \max(\min(\phi^{n+1},1),-1)$.
    \STATE Compute residual $r=\|\phi^{n+1}-\phi^k\|_{L^1(\Omega)}/|\Omega|$.
    \STATE Update $(\phi^k,\psi^k)\gets (\phi^{n+1},\psi^{n+1})$, $k\gets k+1$.
  \UNTIL{$r<\mathrm{tol}$}
  \STATE Accept $(\phi^{n+1},\psi^{n+1})$.
  \IF{AMR enabled}
    \STATE Refine around the interface and conservatively map $(\phi^{n+1},\psi^{n+1})$.
  \ENDIF
\ENDFOR
\end{algorithmic}
\end{algorithm}

\paragraph{Linear solver and preconditioning.}
We solve \eqref{eq:block_system_operator} by restarted GMRES \cite{saad_schultz_1986}
applied to the assembled \emph{block} system, with an ILUC0 preconditioner \cite{meijerink_vdv_1977,saad_2003}.
Across all experiments reported later, the block solve typically converges in 2--3 GMRES iterations per Picard step,
indicating near-optimal behavior for the considered meshes and parameter regimes.

\paragraph{Coupled versus segregated solution strategy.}
Solving the mixed CH system in a fully coupled fashion is significantly more robust than a segregated
(outer $\phi$ / inner $\psi$) approach, especially for stiff interfacial regimes.
This is consistent with observations in energy-stable CH time stepping \cite{eyre_1998}
and with our earlier comparative study for second-order CH formulations \cite{musil_furst_2025_enhanced}.
For partially explicit or loosely split treatments of the fourth-order operator, severe stability restrictions on
$\delta t$ are commonly observed; see, e.g., the discussion in \cite{pavlidis_2020_mg}.

\paragraph{Adaptive mesh refinement (optional).}
When AMR is enabled, we periodically refine a narrow band around the interface (e.g.\ $|\phi|<\phi_{\mathrm{ref}}$)
and conservatively map $\phi$ and $\psi$ onto the new mesh before continuing to the next time step.

\section{Numerical Experiments}\label{sec:experiments}

All experiments employ the quartic double-well potential $W(\phi)=\tfrac{1}{4}(1-\phi^2)^2$
and a quartically degenerate mobility $M(\phi)=(1-\phi^2)^2$ (i.e., $l=2$).
For the ZHOU models ($k=2,8$), we increase the mobility degeneracy to $M(\phi)=(1-\phi^2)^3$
to better satisfy the compatibility condition $1 \le k < l \le 2k+1$ from~\cite{zhou_2025_achic}; note that for $k=8$ this still falls short of the strict requirement $l \ge 8$, but we find empirically that $l=3$ already provides significant improvement. A small regularization offset $\alpha_Q \approx 10^{-6}$ is applied to $Q'(\phi)$ to prevent numerical singularities at $\phi = \pm 1$ while preserving conservation properties.

\subsection{Experiment A: Multi-Scale Droplet Coarsening and Coexistence}\label{sec:exp:droplets}

This study investigates a multi-scale population dynamics problem involving repeated high-curvature events, droplet--droplet interaction, and size-dependent shrinkage. 
The setup tends to amplify systematic conservation defects---most visibly through biased (premature) extinction of the smallest droplet---and is therefore a sensitive discriminator between the second-order NMN formulation and the proposed third-order designs.

\paragraph{Problem setup.}
We consider a two-dimensional domain $\Omega = [0,4]\times[0,1]$ discretized by a uniform Cartesian mesh of $400\times 100$ cells, yielding $\Delta x=\Delta y=10^{-2}$.
The initial condition consists of four circular droplets aligned along the midline $y=0.5$ with progressively decreasing radii $R \in \{0.15, 0.10, 0.06, 0.03\}$ (Fig.~\ref{fig:4droplets_ic}).
The phase field is initialized from a signed-distance function using a clipped linear profile,
$\phi(\mathbf{x},0)=\mathrm{clip}(d(\mathbf{x})/\varepsilon,-1,1)$, to impose a controlled diffuse layer of thickness $\mathcal{O}(\varepsilon)$.
We report results for interface thicknesses $\varepsilon \in \{2\Delta x, 4\Delta x\}$.
Time integration uses adaptive time stepping based on the nonlinear iteration count, with $\Delta t \in [10^{-10}, 5\cdot 10^{-3}]\,\mathrm{s}$ and a target of $\approx 20$ Picard iterations per step. We employed a tight nonlinear convergence tolerance of $tol=10^{-9}$ to effectively isolate the $\mathcal{O}(\varepsilon^3)$ geometric-volume error from solver residuals.

\begin{figure}[h!]
  \centering
  \includegraphics[width=0.95\linewidth]{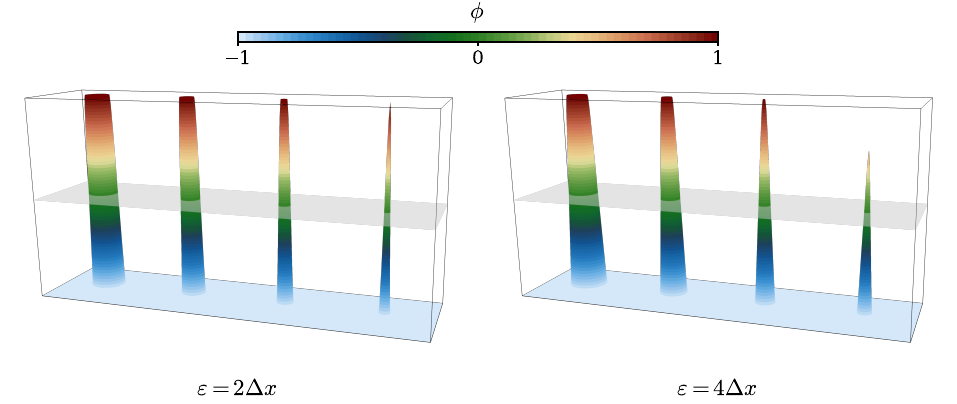}
  \caption{Initial condition for the multi-scale coarsening study: four droplets with radii $R \in \{0.15, 0.10, 0.06, 0.03\}$ on a $4\times 1$ domain.}
  \label{fig:4droplets_ic}
\end{figure}

\vspace{-1em}
\paragraph{Results and Discussion.}
The system evolves via Ostwald ripening, driven by the chemical potential difference between droplets of varying curvature. While the physics dictates that the smallest droplet ($D_4$) should shrink, numerical errors often accelerate this process, distinguishing robust schemes from dissipative ones.
The significance of the proposed $\mathcal{O}(\varepsilon^3)$ accuracy becomes most apparent when the droplet radius $R$ approaches the interface thickness $\varepsilon$. For a small droplet in 2D, the total phase volume scales as $V \sim R^2$. If $R \sim \mathcal{O}(\varepsilon)$, the physical volume is of order $\mathcal{O}(\varepsilon^2)$. Consequently, any numerical volume defect of order $\mathcal{O}(\varepsilon^2)$---typical of standard schemes or the unbalanced NMN model---is comparable in magnitude to the droplet's entire mass. This parasitic loss acts as an artificial sink, causing premature extinction. In contrast, the moment-balanced designs (EXP, PAD\'E) reduce the defect to $\mathcal{O}(\varepsilon^3)$, which remains negligible relative to the physical volume of even marginal droplets.

Figure~\ref{fig:4droplets_volume} confirms this scaling. The baseline M-model (red) exhibits steep volume loss driven by zeroth-order bulk leakage, while NMN (blue) reduces magnitude but retains a persistent drift due to non-zero $\mathcal{C}_1$. In contrast, the third-order EXP and PAD\'E variants (orange/green) yield essentially flat error profiles, minimizing volume jumps during topological changes. Monotonic energy decay (Fig.~\ref{fig:4droplets_energy_decay}) confirms this accuracy preserves the gradient-flow structure; the \enquote{delayed} extinction is a physical restoration of the correct timescale, not a stabilization artifact.

Warp visualizations (Figs.~\ref{fig:4droplets_snap_2dx}--\ref{fig:4droplets_snap_4dx}) reveal the mechanism: while the M-model suffers from bulk \enquote{sagging} ($\phi < 1$), EXP and PAD\'E maintain uniform saturation $\phi \approx \pm 1$. This high-order endpoint degeneracy ($k\ge 2$) confines errors to the interface, where the moment-cancellation mechanism ($\mathcal{C}_1 \approx 0$) eliminates the drift.

\begin{figure}[h!]
  \centering
  \includegraphics[width=\linewidth]{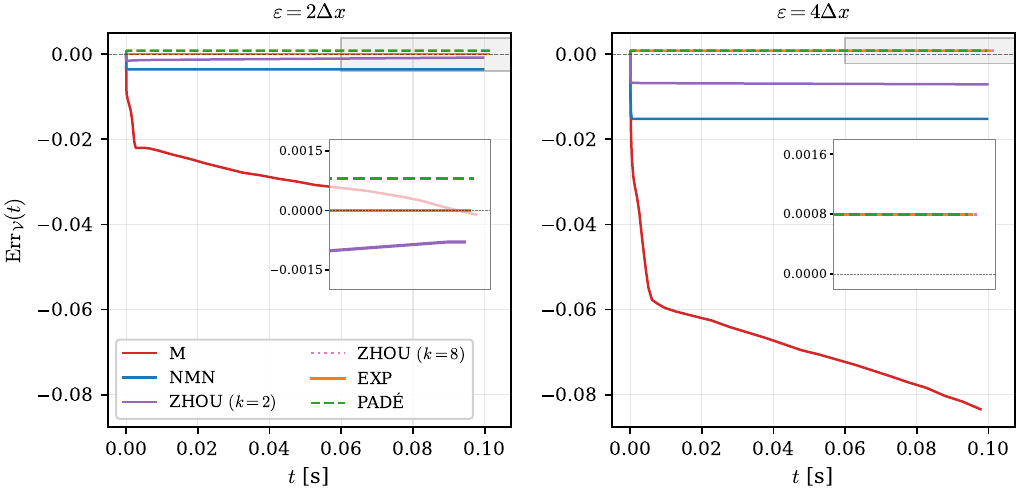}
  \caption{Volume conservation error $\ErrV(t)$ during coarsening.
  }
  \label{fig:4droplets_volume}
\end{figure}

\begin{figure}[h!]
  \centering
  \includegraphics[width=0.9\linewidth]{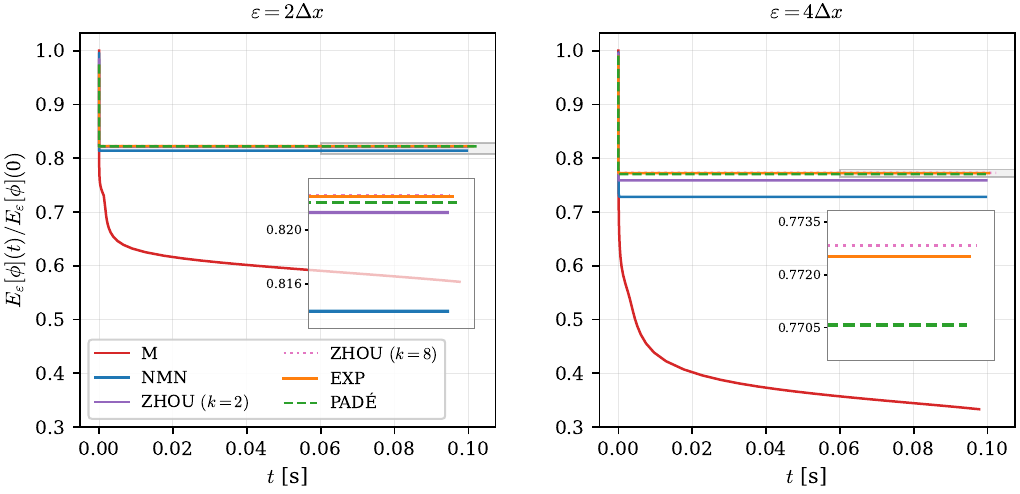}
  \caption{Free-energy decay for the compared conservation mechanisms. 
  }
  \label{fig:4droplets_energy_decay}
\end{figure}

\begin{figure}[p]
  \centering
  \includegraphics[width=\linewidth]{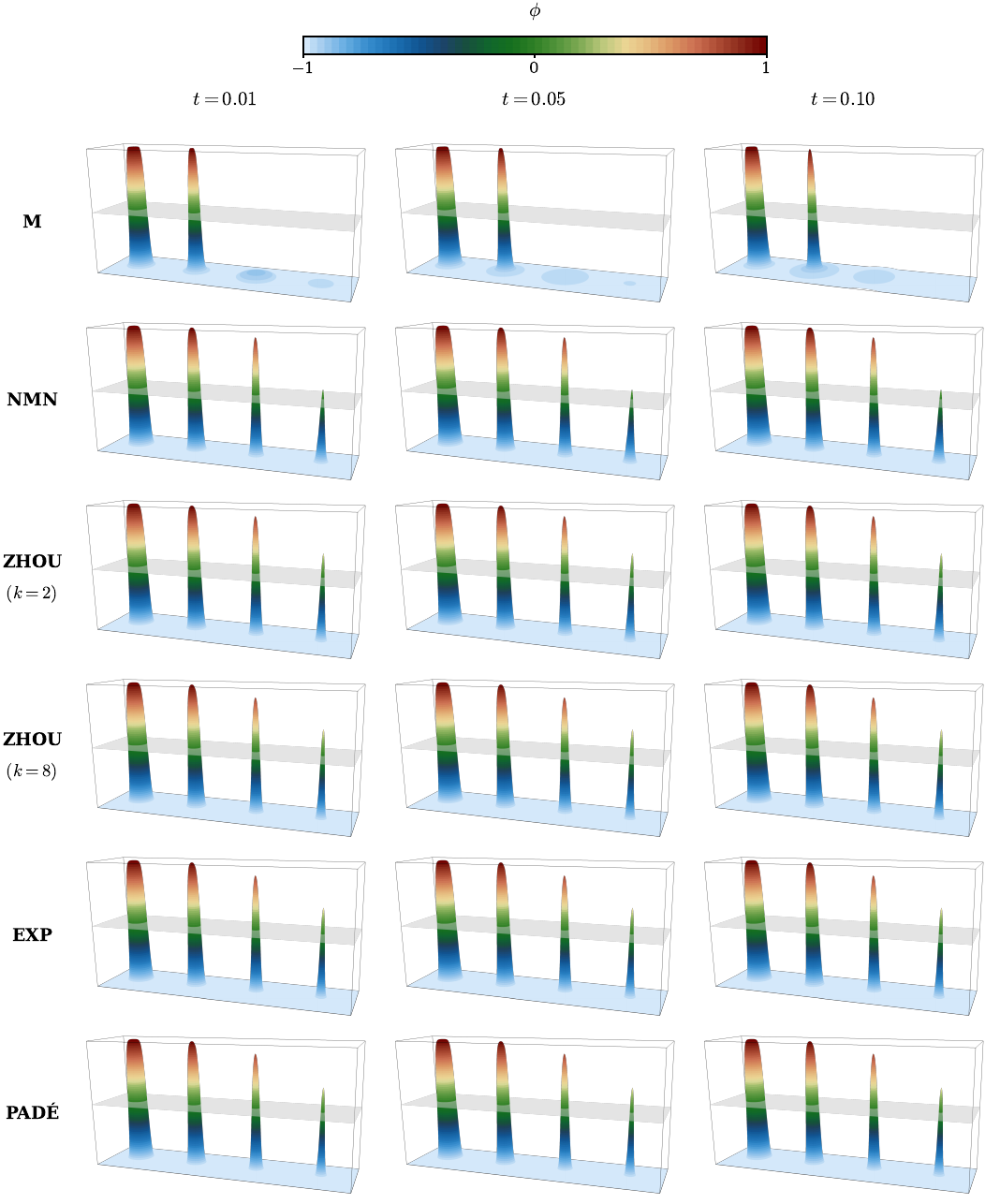}
  \caption{Evolution of the four-droplet system for $\varepsilon=2\Delta x$. The phase field $\phi$ is visualized using a scalar warp, with the $z=0$ plane indicating the interface. The M-model shows bulk sagging (rounded plateaus), whereas EXP and Pad\'e variants maintain flatter, saturated bulk phases ($\phi \approx \pm 1$) and a cleaner interfacial transition.}
  \label{fig:4droplets_snap_2dx}
\end{figure}

\begin{figure}[p]
  \centering
  \includegraphics[width=\linewidth]{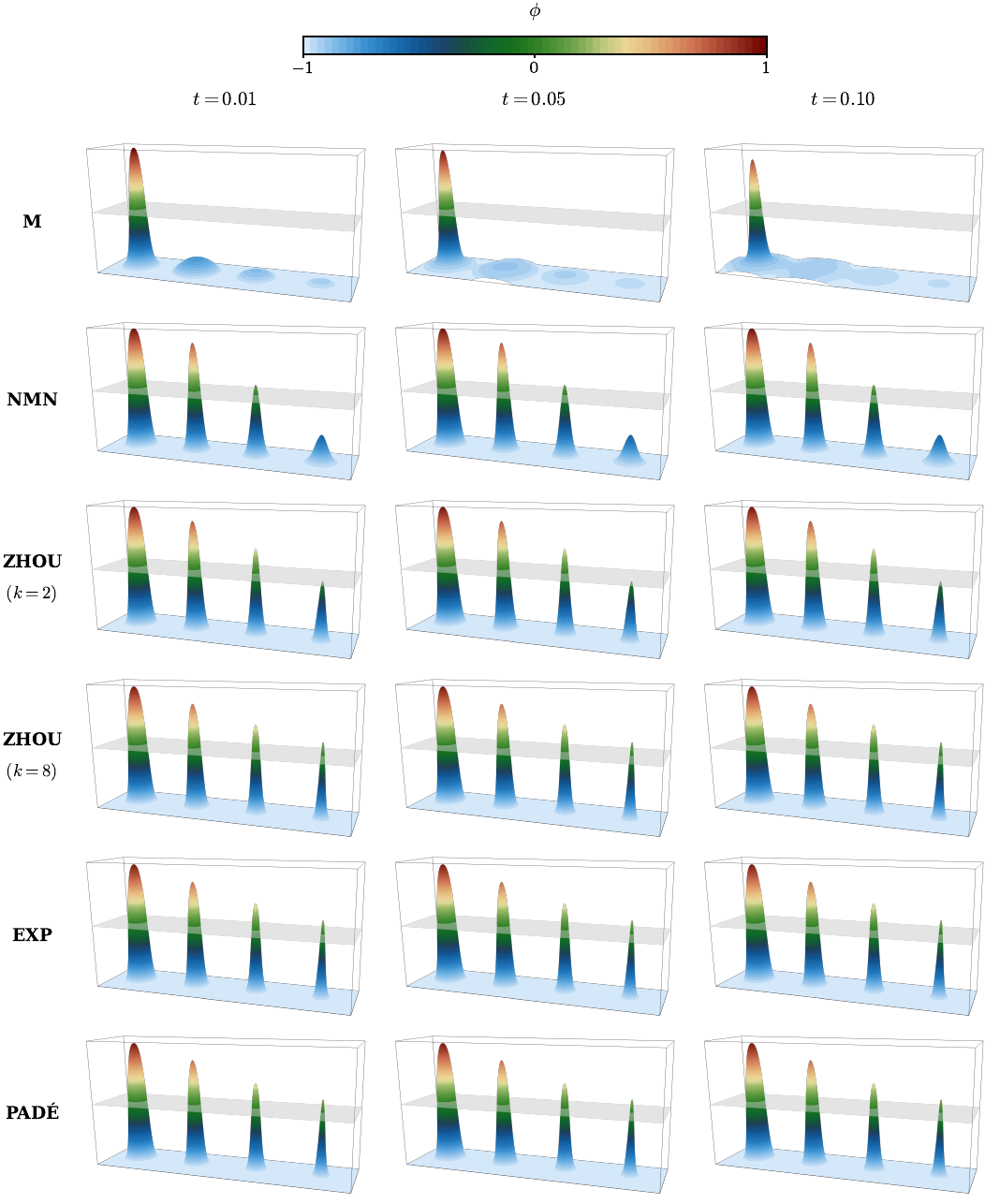}
  \caption{Evolution of the four-droplet system for $\varepsilon=4\Delta x$. Visualization uses scalar warping relative to the zero-level plane. Even with a thicker interface, the third-order models maintain better bulk integrity than the baseline M-model.}
  \label{fig:4droplets_snap_4dx}
\end{figure}

\clearpage
\subsection{Experiment B: Relaxation of High-Curvature \enquote{Flower} Morphology}\label{sec:exp:flower}

This experiment examines the relaxation of a non-convex, high-curvature \enquote{flower} geometry. The shape induces rich transient dynamics: high positive curvature at the tips drives rapid retraction, while negative curvature in the valleys slows the smoothing process. These localized fluxes provide a stringent test of bulk saturation and volume accuracy.

\paragraph{Problem setup.}
We solve the evolution on the unit square $\Omega=[0,1]^2$ with a $200\times200$ mesh ($\Delta x=5\cdot10^{-3}$).
The initial condition is a six-petaled flower shape (Fig.~\ref{fig:flower_ic}) initialized from a clipped signed-distance function.
We perform a sensitivity study over interface widths $\varepsilon \in \{2,3,4\}\Delta x$.
All variants use the same quartic double-well potential $W(\phi)=\tfrac14(\phi^2-1)^2$, quartically degenerate mobility $M(\phi)=(1-\phi^2)^2$, and the same semi-implicit stabilization parameter (as in Section~\ref{sec:numerics}); the only variable is the conservation kernel $Q(\phi)$ (M, NMN, EXP, Pad\'e).

\begin{figure}[H]
  \centering
  \begin{minipage}[c]{0.48\linewidth}
    \centering
    \includegraphics[width=\linewidth]{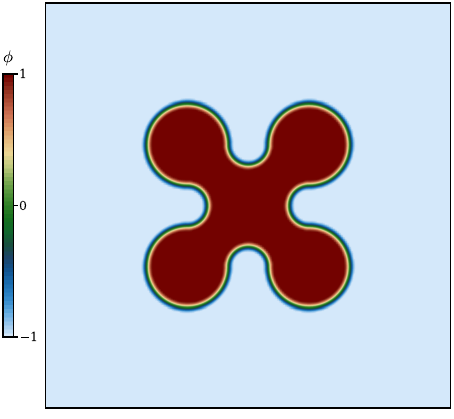}
    \caption{Initial \enquote{flower} geometry on $\Omega=[0,1]^2$, featuring regions of high positive and negative curvature.}
    \label{fig:flower_ic}
  \end{minipage}\hfill
  \begin{minipage}[c]{0.48\linewidth}
    \centering
    \includegraphics[width=\linewidth]{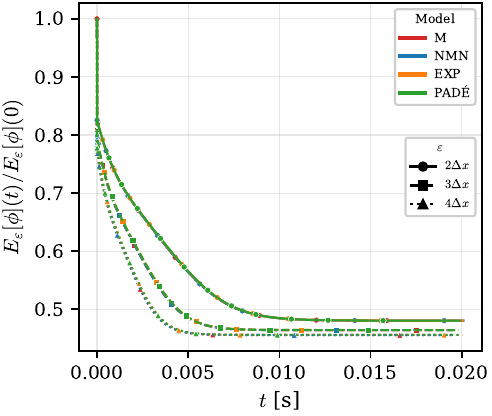}
    \caption{Energy decay $E(t)$ during relaxation. The rapid initial drop corresponds to smoothing of the flower tips.}
    \label{fig:flower_energy_decay}
  \end{minipage}
\end{figure}

\paragraph{Results and Discussion.}
The relaxation of the ``flower'' shape separates the error dynamics into two distinct regimes: a fast, high-curvature initial transient ($t < 0.01$) and a slow, quasi-circular drift ($t > 0.01$). This separation allows us to isolate the impact of the dynamic moment $\mathcal{J}_1$.
The early evolution is dominated by regions of large mean curvature $H$, where the tips retract and valleys bulge. According to the error expansion $\ErrV \propto \varepsilon^2 \int_\Gamma H \,\mathcal{C}_1$, this is where the interfacial defect is maximized.
Figure~\ref{fig:flower_volume} clearly demonstrates the efficacy of the moment balance. The M-model (red) suffers an immediate, massive drop in volume during the first few time steps, driven by bulk leakage. The NMN model (blue) corrects the bulk leakage but still exhibits a visible initial ``jump'' and oscillation, corresponding to the non-zero interfacial coefficient $\mathcal{C}_1 \approx -0.645$.
In striking contrast, the EXP and PAD\'E models (orange/green) are virtually insensitive to this high-curvature transient. The error curves remain near zero even during the most violent shape changes, proving that the condition $\mathcal{M}_1 + \mathcal{J}_1 \approx 0$ successfully cancels the leading geometry-dependent error.

Comparing the panels for $\varepsilon=2\Delta x$ through $\varepsilon=4\Delta x$ in Fig.~\ref{fig:flower_volume}, the lower-order models show a strong sensitivity to resolution. The volume drift in the M-model roughly doubles as the interface thickens. The third-order designs are remarkably robust: their error curves are nearly identical across all three resolutions. This aligns with the theoretical prediction that the error is pushed to $\mathcal{O}(\varepsilon^3)$; even at coarse resolutions ($\varepsilon=4\Delta x$), the third-order term remains small enough to prevent significant drift.

\begin{figure}[h]
  \centering
  \includegraphics[width=\linewidth]{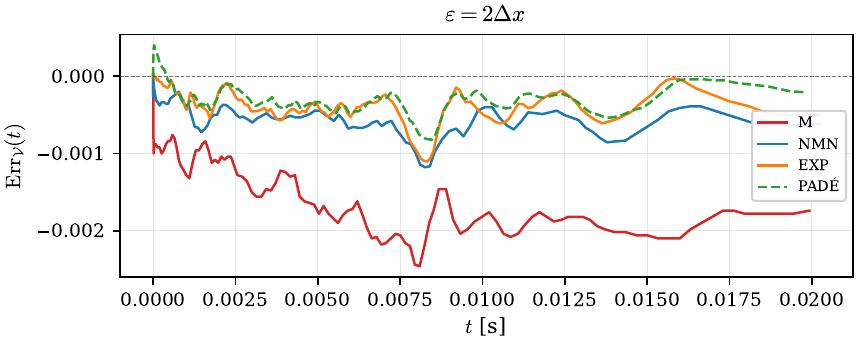}
  \includegraphics[width=\linewidth]{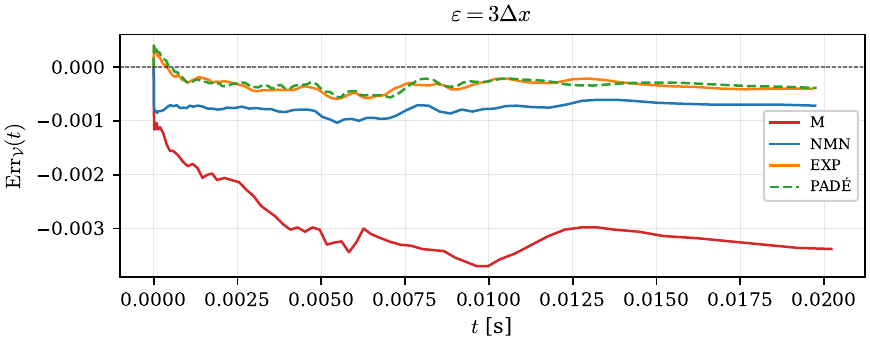}
  \includegraphics[width=\linewidth]{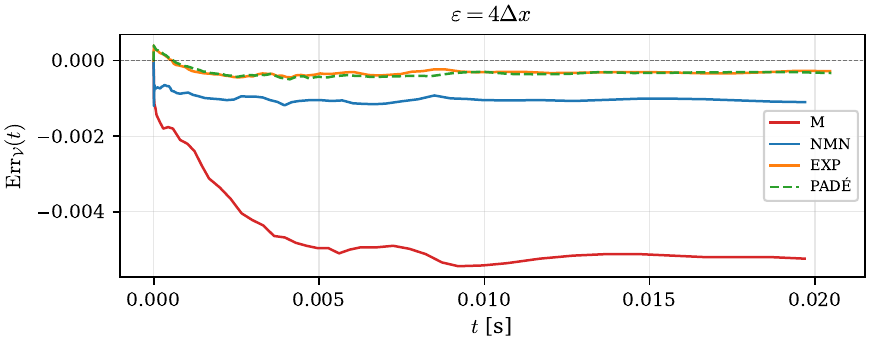}
  \caption{Volume conservation error $\ErrV(t)$ for the flower relaxation across three interface widths ($\varepsilon/\Delta x \in \{2,3,4\}$). The systematic hierarchy of conservation quality is evident: M $<$ NMN $<$ EXP/Pad\'e.}
  \label{fig:flower_volume}
\end{figure}

\begin{table}[h]
\centering
\caption{Computational cost and solver statistics for the flower relaxation ($T=0.02$\,s). The third-order models (EXP, Pad\'e) incur negligible overhead compared to NMN while providing superior conservation.\vspace{0.4em}}
\label{tab:flower_cpu}
\begin{tabular}{llcccc}
\toprule
\(\varepsilon/\Delta x\) & Model & Wall time [s] & Avg. \(\Delta t\) [\(\mu\)s] & Avg. Picard iters & Avg. GMRES iters \\
\midrule
2 & M & 1084.3 & 164.16 & 19.21 & 8.39 \\
2 & NMN $(k=1)$ & 1094.0 & 141.20 & 20.14 & 9.44 \\
2 & EXP $(k=1)$ & 1175.7 & 122.80 & 20.11 & 9.03 \\
2 & PAD\'E $(k=1)$ & 1245.9 & 107.14 & 19.59 & 8.18 \\
\midrule
3 & M & 600.2 & 290.31 & 17.41 & 9.91 \\
3 & NMN $(k=1)$ & 885.8 & 132.30 & 19.88 & 9.23 \\
3 & EXP $(k=1)$ & 832.6 & 187.91 & 18.85 & 6.34 \\
3 & PAD\'E $(k=1)$ & 941.0 & 165.25 & 18.58 & 6.96 \\
\midrule
4 & M & 412.2 & 396.54 & 15.08 & 8.95 \\
4 & NMN $(k=1)$ & 970.5 & 100.98 & 19.57 & 8.79 \\
4 & EXP $(k=1)$ & 783.8 & 247.44 & 17.84 & 6.82 \\
4 & PAD\'E $(k=1)$ & 924.1 & 176.42 & 18.53 & 6.09 \\
\bottomrule
\end{tabular}
\end{table}

Table~\ref{tab:flower_volume} quantifies the accumulated late-time error. The EXP and PAD\'E kernels achieve a $\sim 10\times$ reduction in error compared to the baseline M-model and a $\sim 2\times$ improvement over NMN. Crucially, while NMN volume loss accumulates linearly over time (constant negative slope), the balanced kernels show a negligible trend.
Regarding computational cost, Table~\ref{tab:flower_cpu} indicates that the third-order kernels incur only minor overhead. It should be noted that the reported wall times are not definitive benchmarks, as they depend heavily on the specific settings of the adaptive time-stepping strategy (e.g., the target number of Picard iterations, set here to $\approx 20$), the nonlinear tolerance, and the predictor used for the initial guess $\phi^0$. However, under identical solver configurations, the EXP and PAD\'E models behave similarly to the standard NMN model, confirming that the improved conservation properties do not require a fundamentally more expensive solution procedure.

\begin{figure}[p]
  \centering
  \includegraphics[width=\linewidth]{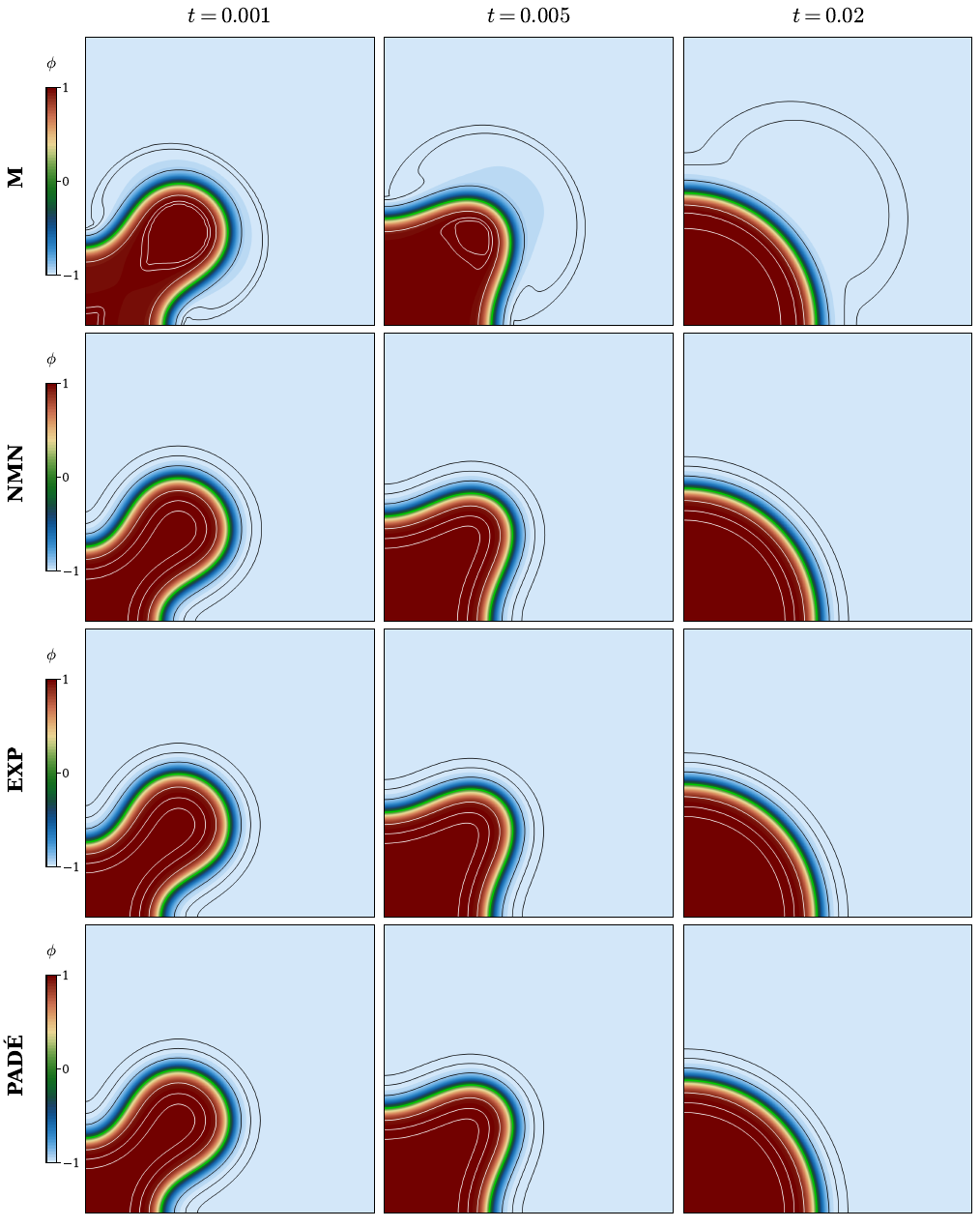}
  \caption{Flower relaxation snapshots for $\varepsilon=2\Delta x$. \textbf{Visual overlay:} White contours indicate bulk saturation levels $\phi \in \{0.9, 0.99, 0.999\}$; black contours indicate $\phi \in \{-0.9, -0.99, -0.999\}$. The tight contour grouping in the EXP/Pad\'e results demonstrates superior bulk preservation compared to the diffuse leakage seen in the M-model.}
  \label{fig:flower_snap_2dx}
\end{figure}

\begin{figure}[p]
  \centering
  \includegraphics[width=\linewidth]{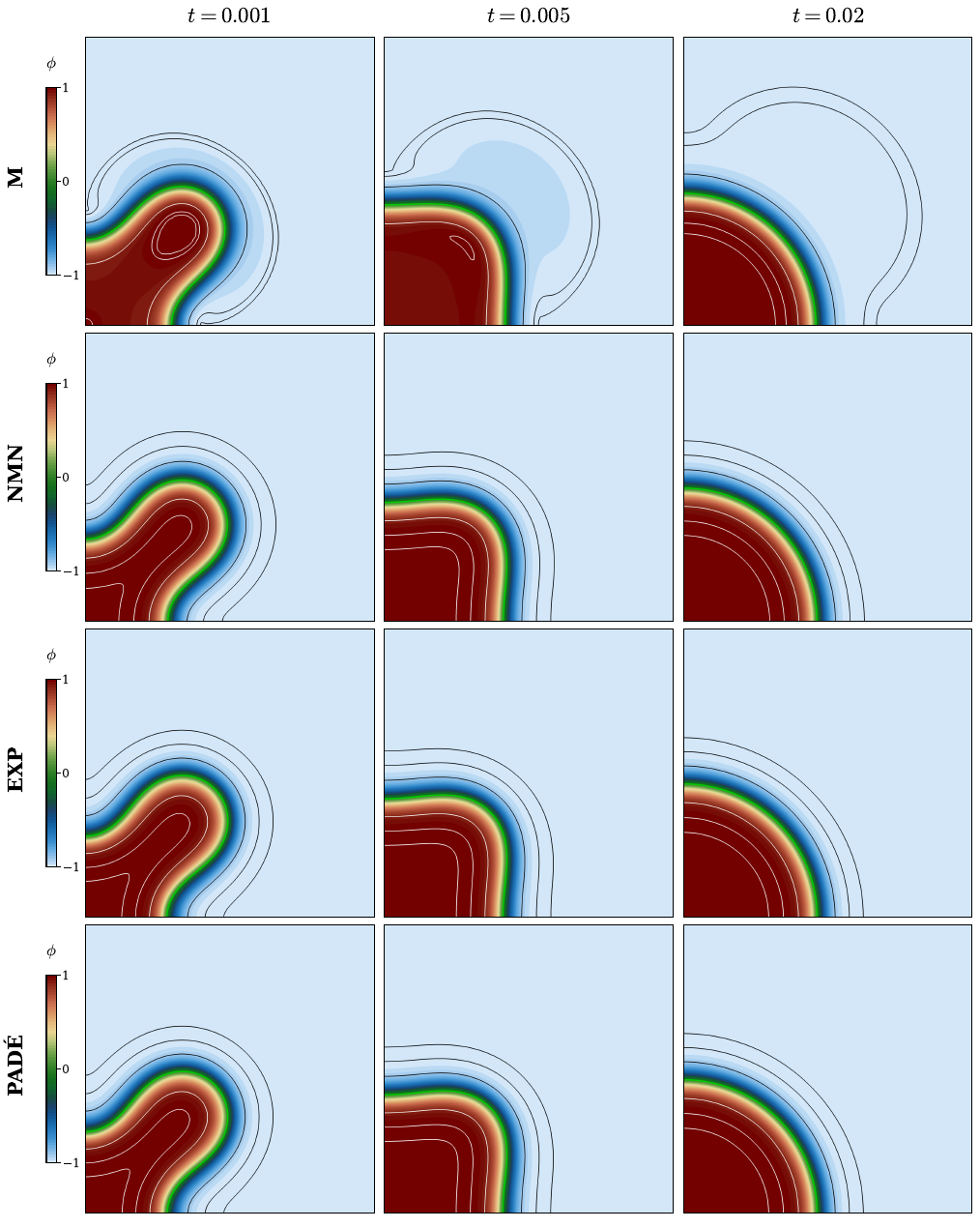}
  \caption{Flower relaxation snapshots for $\varepsilon=3\Delta x$. Contours: White=$\phi \in \{0.9, 0.99, 0.999\}$, Black=$\phi \in \{-0.9, -0.99, -0.999\}$. The third-order models maintain sharper saturation profiles as the interface width increases.}
  \label{fig:flower_snap_3dx}
\end{figure}

\begin{figure}[p]
  \centering
  \includegraphics[width=\linewidth]{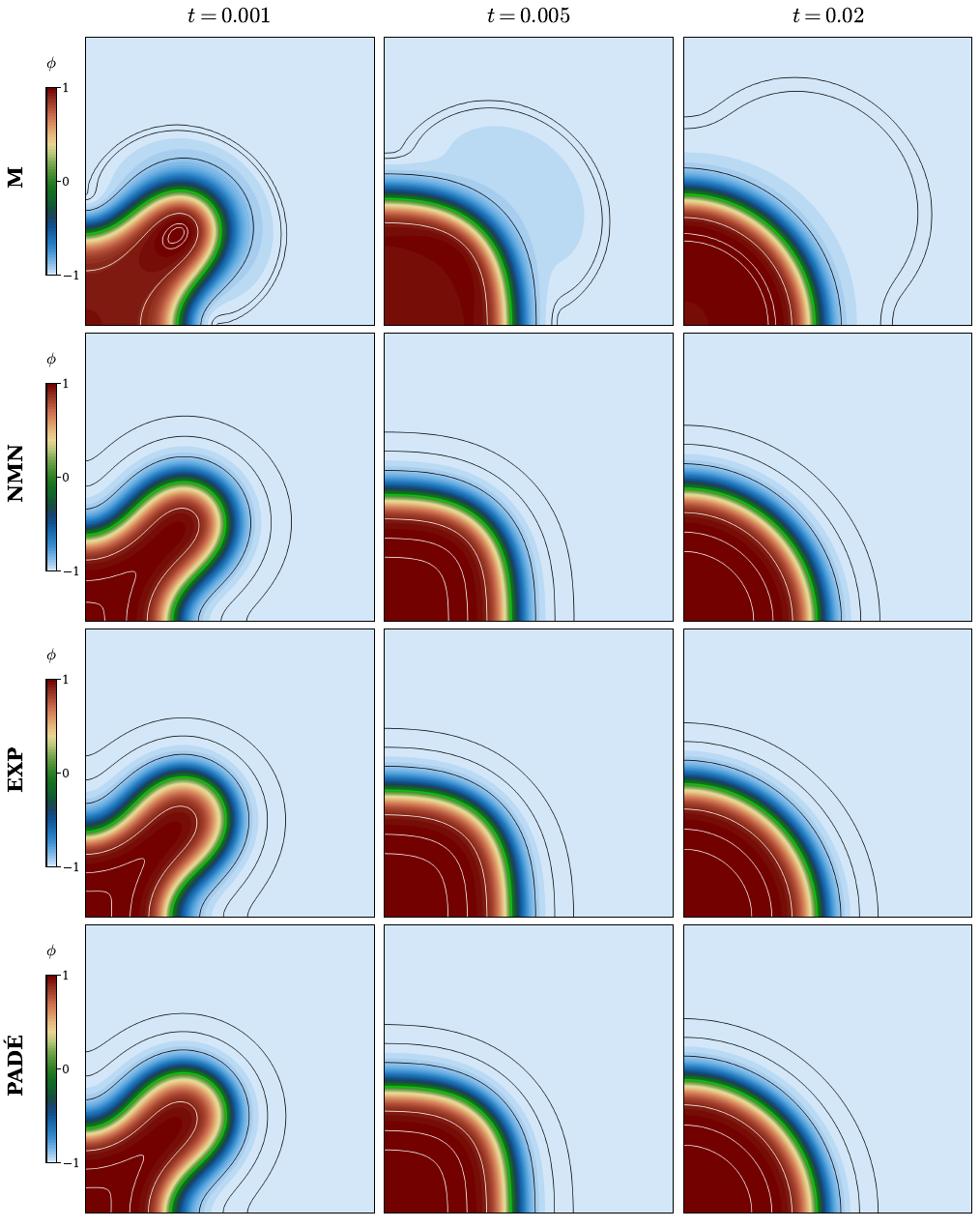}
  \caption{Flower relaxation snapshots for $\varepsilon=4\Delta x$. Contours: White=$\phi \in \{0.9, 0.99, 0.999\}$, Black=$\phi \in \{-0.9, -0.99, -0.999\}$. Despite the coarse interface, the Pad\'e and EXP kernels prevent the spurious bulk gradients observed in the M-model.}
  \label{fig:flower_snap_4dx}
\end{figure}

\begin{table}[h]
\centering
\caption{Late-time volume conservation error $|\ErrV|$ for the flower relaxation. High-order models
(EXP, Pad\'e) achieve nearly $\varepsilon$-independent error and $\sim$10$\times$ improvement over M.\vspace{0.4em}}
\label{tab:flower_volume}
\begin{tabular}{lcccc}
\toprule
Model & $\varepsilon=2\Delta x$ & $\varepsilon=3\Delta x$ & $\varepsilon=4\Delta x$ & Avg. improvement \\
\midrule
M              & $1.92\times10^{-3}$ & $3.29\times10^{-3}$ & $5.22\times10^{-3}$ & (baseline) \\
NMN $(k=1)$    & $5.96\times10^{-4}$ & $7.76\times10^{-4}$ & $9.89\times10^{-4}$ & $4.3\times$ \\
EXP $(k=1)$    & $3.88\times10^{-4}$ & $3.49\times10^{-4}$ & $3.53\times10^{-4}$ & $9.7\times$ \\
PAD\'E $(k=1)$ & $3.62\times10^{-4}$ & $3.77\times10^{-4}$ & $3.82\times10^{-4}$ & $9.2\times$ \\
\bottomrule
\end{tabular}
\end{table}

\clearpage
\section{Conclusions and outlook}

We clarified the mechanism governing \emph{geometric} volume drift in conservation-improved
Cahn--Hilliard (CH--IC) models by separating two contributions that are often conflated:
(i) a \emph{bulk} channel controlled by the endpoint degeneracy of the conserved mapping $Q$,
and (ii) an \emph{interfacial} channel controlled by an explicit combined moment
$\mathcal{C}_1[Q]=\mathcal{M}_1[Q]+\mathcal{J}_1[Q]$ arising from the Jacobian correction and the curvature-induced profile
distortion $\Phi_1$. For odd $Q$, the $\mathcal{O}(\varepsilon)$ interfacial contribution cancels by symmetry,
so the leading interfacial defect is $\mathcal{O}(\varepsilon^2)$ and admits closed one-dimensional
quadrature formulas for $\mathcal{M}_1$ and $\mathcal{J}_1$, enabling a practical inverse-design workflow.

The resulting design logic is transparent: (a) enforce sufficient endpoint degeneracy of $Q'$ so that the
\emph{bulk} contribution is at most $\mathcal{O}(\varepsilon^3)$, and (b) tune the kernel shape so that the
\emph{interfacial} coefficient $\mathcal{C}_1[Q]$ vanishes, promoting geometric-volume accuracy to
$\mathcal{O}(\varepsilon^3)$ at the model level. In the numerical benchmarks, the moment-balanced shaped
kernels (EXP and PAD\'E) deliver nearly $\varepsilon$-independent late-time geometric-volume error and
an order-of-magnitude improvement over the baseline mass-conserving model, while maintaining robust
energy dissipation and negligible additional computational overhead.

\paragraph{Outlook.}
Several directions appear particularly promising:

\begin{itemize}
  \item \textbf{Improved nonlinear solvers beyond baseline Picard.}
  While the present block-coupled Picard iteration is robust, sharper kernel degeneracies and stiff
  high-curvature transients motivate improved fixed-point strategies (e.g.\ Anderson-accelerated Picard,
  quasi-Newton updates, or physics-informed block preconditioners) that preserve the discrete
  $Q^{n+1}-Q^n$ conservation update and the coercive Eyre-type stabilization.

  \item \textbf{Restricted-energy (V-model) within the same moment framework.}
  The doubly-degenerate restricted-energy formulation of Salvalaglio \emph{et al.} \cite{salvalaglio_2021} fits naturally into the
  present bulk/interface decomposition. Using
  \[
    E_\varepsilon^{V}[\phi] \;=\; \int_\Omega g_0(\phi)\left(\frac{1}{\varepsilon}F(\phi)
    + \frac{\varepsilon}{2}|\nabla\phi|^2\right)\,dx,
    \qquad g_0(\phi)=\frac{1}{\gamma(1-\phi^2)^p},
  \]
  one finds that the leading inner profile remains the standard $\tanh$ heteroclinic, while the
  curvature correction $\Phi_1$ changes through a weighted Sturm--Liouville operator. For the
  \emph{mass} proxy $Q(\phi)=\phi$, our analysis yields a unique ``sweet spot''
  $p^\star\approx 1.5$ at which the leading \emph{interfacial} $\mathcal{O}(\varepsilon^2)$ defect cancels.
  A key next step is to couple the restricted-energy idea with a $Q$-constraint (Onsager form),
  e.g.\ with $Q'(\phi)\propto (1-\phi^2)^k$ (including $k=1$ with endpoint-vanishing shaping),
  so that \emph{bulk} and \emph{interfacial} error channels can be suppressed simultaneously by a joint
  tuning of $(p,\;Q)$.
\item \textbf{Skew (commutator) flux form for Onsager CH--IC.}
A promising direction is to re-express the Onsager CH--IC dynamics in a \emph{skew/commutator} flux form that avoids placing the scaled chemical potential $\mu/Q'(\phi)$ inside a gradient. In the variational formulation one has the conservative law
$\partial_t Q(\phi)+\nabla\!\cdot J=0$ with $J=-M(\phi)\nabla(\mu/Q'(\phi))$.
If the mobility is chosen compatibly as $M(\phi)=(Q'(\phi))^2\widetilde M(\phi)$, then the flux admits the exact identity
\[
J=-\widetilde M(\phi)\Big(Q'(\phi)\nabla\mu-\mu\nabla Q'(\phi)\Big),
\]
so that no division by $Q'$ appears \emph{inside} spatial derivatives. This purely algebraic rewrite may be advantageous for conservative CH--NS coupling and for fully implicit discretizations, and it motivates studying which mobility--kernel pairs preserve the desired dissipation and high-order geometric-volume properties within the same moment framework.
\end{itemize}

\section*{Acknowledgments}
The author gratefully acknowledge support form Institute of Thermomechanics, Czech Academy of Sciences (RVO:61388998). He would also like to acknowledge institutional support from Czech Technical University in Prague (SGS25/123/OHK2/3T/12).

\appendix

\section{Matched Asymptotic Analysis}
\label{app:A}

This appendix provides a detailed matched-asymptotic derivation of the $\varepsilon\to 0$ limit of the unscaled CH--IC model, yielding the corresponding sharp-interface evolution law. The analysis follows the framework of Cahn \emph{et al.}~\cite{cahn_1996}, Pego~\cite{pego_1989}, and Bretin \emph{et al.}~\cite{bretin_2022}, but keeps the physical time variable without slow-time rescaling. As a result, the inner expansion yields a leading quasi-stationarity constraint $V_0=0$, so interface motion enters at $V=\mathcal{O}(\varepsilon)$; the first solvability condition then produces the surface-diffusion law. This ``physical-time'' choice is the natural one for the full hydrodynamic NSCH system, where the Navier--Stokes advection and inertia are posed in the same time variable at finite $\varepsilon$. Moreover, as noted in the main-text remark on mobility prefactor scaling, choosing a prefactor $M_\ast\sim\varepsilon^{a}$ effectively rescales the time scale of the CH relaxation and can shift the apparent velocity hierarchy (e.g.\ recovering an $\mathcal{O}(1)$ interfacial speed in coupled settings without changing the matched-asymptotic structure).

\subsection{Governing equations}
We consider the generalized Cahn--Hilliard dynamics on a bounded domain $\Omega\subset\mathbb{R}^d$,
$d\in\{2,3\}$:
\begin{align}
\partial_t\phi &= N(\phi)\,\nabla\!\cdot\!\Big(M(\phi)\,\nabla\!\big(N(\phi)\,\mu\big)\Big),
\label{eq:A_evol}\\[2pt]
\mu &= \frac{1}{\varepsilon}W'(\phi)-\varepsilon\,\Delta\phi,
\label{eq:A_chem}
\end{align}
where $\varepsilon\ll1$ is the interfacial thickness. We use the standard double-well
$W(\phi)=\tfrac14(1-\phi^2)^2$, the degenerate mobility $M(\phi)=(1-\phi^2)^2$ (so $M(\pm1)=0$), and the
metric factor $N(\phi)=1/Q'(\phi)$, where $Q$ is odd, strictly increasing, and satisfies $Q(\pm1)=\pm1$.

\subsection{Geometry, inner variables, and operator identities}
\label{subsec:geometric_setup}
Let $\Gamma(t)$ be the interface separating $\Omega^+(t)$ and $\Omega^-(t)$, with
$\overline{\Omega}=\overline{\Omega^+(t)}\cup\overline{\Omega^-(t)}$ and
$\Gamma(t)=\overline{\Omega^+(t)}\cap\overline{\Omega^-(t)}$.
Assume $\Gamma(t)$ is a smooth, closed $(d-1)$-dimensional hypersurface evolving smoothly in time.
Let $d(x,t)$ denote the signed distance to $\Gamma(t)$ (positive in $\Omega^+(t)$), and define the unit normal
$\nu=\nabla d$ pointing into $\Omega^+(t)$. There exists $\delta_0>0$ such that $d$ is smooth in the tubular
neighborhood $\mathcal{T}_{\delta_0}(t)=\{x:\,|d(x,t)|<\delta_0\}$ and each $x\in\mathcal{T}_{\delta_0}(t)$ admits a unique
projection $\pi(x,t)\in\Gamma(t)$ with
\[
x=\pi(x,t)+d(x,t)\,\nu(\pi(x,t),t).
\]
We set the stretched normal coordinate $z=d(x,t)/\varepsilon$ and use local surface coordinates $s$ on $\Gamma(t)$.
In the inner region, the fields are represented as
\[
\Phi(z,s,t):=\phi(x,t),\qquad \Upsilon(z,s,t):=\mu(x,t).
\]

\paragraph{Sign convention.}
In Onsager form one may define the physical flux as $J_{\mathrm{Ons}}=-\,M(\phi)\nabla\!\big(N(\phi)\mu\big)$.
In what follows we work with the signless quantity $J:=M(\phi)\nabla\!\big(N(\phi)\mu\big)$, since only $\nabla\!\cdot J$
enters \eqref{eq:A_evol} and all matching relations compare inner and outer expressions within the same convention.

\paragraph{Signed-distance identities.}
We collect the differential-geometric identities for the signed distance function used in the derivation, following
\cite{bretin_2022,ambrosio_2000}; for readability we write explicit formulas in $d=2$, with the extension to $d=3$
obtained by replacing $(\partial_s,\partial_{ss})$ by surface derivatives.
The eikonal relation $|\nabla d|=1$ implies $\nabla d=\nu$. Moreover,
\begin{equation}
\Delta d(x,t)=\sum_{k=1}^{d-1}\frac{\kappa_k(\pi(x,t),t)}{1+d(x,t)\,\kappa_k(\pi(x,t),t)},
\label{eq:laplacian_d_general}
\end{equation}
where $\kappa_k$ are the principal curvatures and $H=\sum_{k=1}^{d-1}\kappa_k$ is the mean curvature on $\Gamma(t)$.
In $d=2$, with arc-length parameter $s$ on $\Gamma(t)$,
\begin{equation}
\Delta d=\frac{H}{1+\varepsilon z H},
\qquad
\nabla S=\frac{\tau}{1+\varepsilon z H},
\qquad
\Delta S=-\frac{\varepsilon z\,\partial_s H}{(1+\varepsilon z H)^3},
\label{eq:d2_distance_identities}
\end{equation}
where $\tau$ is the unit tangent and $S(x,t)$ is the lifted arc-length coordinate.

\paragraph{Transformation of operators (in $d=2$).}
Using $\Phi(z,s,t)=\phi(x,t)$ with $z=d/\varepsilon$ and $s=S(x,t)$, one obtains
\begin{align}
\nabla\phi &=
\frac{1}{\varepsilon}\,\nu\,\partial_z\Phi+\frac{\tau}{1+\varepsilon zH}\,\partial_s\Phi,
\label{eq:grad_transform_compact}\\[2pt]
\Delta\phi &=
\frac{1}{\varepsilon^2}\partial_{zz}\Phi
+\frac{H}{\varepsilon(1+\varepsilon zH)}\,\partial_z\Phi
+\frac{1}{(1+\varepsilon zH)^2}\partial_{ss}\Phi
-\frac{\varepsilon z\,\partial_sH}{(1+\varepsilon zH)^3}\,\partial_s\Phi,
\label{eq:laplacian_full_compact}
\end{align}
and for the normal velocity $V(s,t)=\partial_t X_0(s,t)\cdot\nu(s,t)$,
\begin{equation}
\partial_t\phi
=
-\frac{V}{\varepsilon}\,\partial_z\Phi
+\partial_t^\Gamma\Phi
+(\partial_t S)\,\partial_s\Phi,
\label{eq:time_derivative_compact}
\end{equation}
where $\partial_t^\Gamma$ denotes the time derivative at fixed $(z,s)$.
For the signless quantity $J=M(\phi)\nabla(N(\phi)\mu)$ introduced above, its divergence reads
\begin{equation}
\nabla\!\cdot J
=
\frac{1}{\varepsilon^2}\partial_z\!\Big(M(\Phi)\,\partial_z\big(N(\Phi)\Upsilon\big)\Big)
+\frac{1}{\varepsilon}\frac{H}{1+\varepsilon zH}\,M(\Phi)\,\partial_z\big(N(\Phi)\Upsilon\big)
+\mathcal{T}_1,
\label{eq:div_flux_compact}
\end{equation}
with the tangential remainder
\begin{equation}
\mathcal{T}_1
=
\frac{1}{(1+\varepsilon zH)^2}\partial_s\!\Big(M(\Phi)\,\partial_s\big(N(\Phi)\Upsilon\big)\Big)
-\frac{\varepsilon z\,\partial_sH}{(1+\varepsilon zH)^3}\,M(\Phi)\,\partial_s\big(N(\Phi)\Upsilon\big).
\label{eq:T1_def}
\end{equation}

\paragraph{Asymptotic expansions and matching.}
We postulate outer expansions
\begin{align}
\phi(x,t) &= \phi_0(x,t)+\varepsilon\phi_1(x,t)+\varepsilon^2\phi_2(x,t)+\cdots,
\label{eq:outer_phi_exp_compact}\\
\mu(x,t) &= \varepsilon^{-1}\mu_{-1}(x,t)+\mu_0(x,t)+\varepsilon\mu_1(x,t)+\cdots,
\label{eq:outer_mu_exp_compact}
\end{align}
and inner expansions
\begin{align}
\Phi(z,s,t) &= \Phi_0(z)+\varepsilon\Phi_1(z,s,t)+\varepsilon^2\Phi_2(z,s,t)+\cdots,
\label{eq:inner_Phi_exp_compact}\\
\Upsilon(z,s,t) &= \varepsilon^{-1}\Upsilon_{-1}(z,s,t)+\Upsilon_0(z,s,t)+\varepsilon\Upsilon_1(z,s,t)+\cdots,
\label{eq:inner_Upsilon_exp_compact}
\end{align}
together with $V=V_0+\varepsilon V_1+\varepsilon^2V_2+\cdots$.
Matching in the overlap region is imposed via Van Dyke's principle:
\begin{equation}
\lim_{z\to\pm\infty}\Phi(z,s,t)=\lim_{d\to0^\pm}\phi(x,t),
\qquad
\lim_{z\to\pm\infty}\Upsilon(z,s,t)=\lim_{d\to0^\pm}\mu(x,t),
\label{eq:matching_compact}
\end{equation}
and similarly for flux components.

\subsection{Full inner system (exact curvature factors)}
\label{subsec:full_inner_system}
Substituting \eqref{eq:time_derivative_compact}--\eqref{eq:div_flux_compact} into \eqref{eq:A_evol}--\eqref{eq:A_chem}
yields the full inner problem:
\begin{align}
\Big(-\frac{V}{\varepsilon}\partial_z+\partial_t^\Gamma+(\partial_tS)\partial_s\Big)\Phi
&=
N(\Phi)\Big[
\frac{1}{\varepsilon^2}\partial_z\!\Big(M(\Phi)\,\partial_z\big(N(\Phi)\Upsilon\big)\Big)
+\frac{1}{\varepsilon}\frac{H}{1+\varepsilon zH}\,M(\Phi)\,\partial_z\big(N(\Phi)\Upsilon\big)
+\mathcal{T}_1\Big],
\label{eq:inner_evol_full_compact}\\[4pt]
\Upsilon
&=
\frac{1}{\varepsilon}W'(\Phi)
-\frac{1}{\varepsilon}\partial_{zz}\Phi
-\frac{H}{1+\varepsilon zH}\partial_z\Phi
-\varepsilon\,\mathcal{T}_2,
\label{eq:inner_chem_full_compact}
\end{align}
where $\mathcal{T}_1$ is given by \eqref{eq:T1_def} and
\begin{equation}
\mathcal{T}_2
=
\frac{1}{(1+\varepsilon zH)^2}\partial_{ss}\Phi
-\frac{\varepsilon z\,\partial_sH}{(1+\varepsilon zH)^3}\partial_s\Phi.
\label{eq:T2_def_compact}
\end{equation}
The subsequent order-by-order expansion of \eqref{eq:inner_evol_full_compact}--\eqref{eq:inner_chem_full_compact}
provides the solvability conditions that determine the interfacial dynamics.

\subsection{Order-by-order derivation}
\label{subsec:order_by_order}
We now expand the full inner system \eqref{eq:inner_evol_full_compact}--\eqref{eq:inner_chem_full_compact} in powers of
$\varepsilon$ and collect terms order by order. Throughout we use the inner expansions
\begin{align}
\Phi(z,s,t) &= \Phi_0(z)+\varepsilon\Phi_1(z,s,t)+\varepsilon^2\Phi_2(z,s,t)+\cdots,
\label{eq:Phi_exp_ob}\\
\Upsilon(z,s,t) &= \varepsilon^{-1}\Upsilon_{-1}(z,s,t)+\Upsilon_0(z,s,t)+\varepsilon\Upsilon_1(z,s,t)+\cdots,
\label{eq:Ups_exp_ob}
\end{align}
and $V=V_0+\varepsilon V_1+\varepsilon^2V_2+\cdots$.

\paragraph{Nonlinear compositions.}
For any smooth function $F$, Taylor expansion about $\Phi_0$ yields
\begin{equation}
F(\Phi)
=
F(\Phi_0)
+\varepsilon\,F'(\Phi_0)\,\Phi_1
+\varepsilon^2\!\left(
F'(\Phi_0)\,\Phi_2+\frac12 F''(\Phi_0)\,\Phi_1^2
\right)
+\mathcal{O}(\varepsilon^3).
\label{eq:comp_rule_ob}
\end{equation}
In particular we write
\begin{align}
M(\Phi) &= M_0+\varepsilon M_1+\varepsilon^2 M_2+\mathcal{O}(\varepsilon^3),
\label{eq:M_expand_ob}\\
N(\Phi) &= N_0+\varepsilon N_1+\varepsilon^2 N_2+\mathcal{O}(\varepsilon^3),
\label{eq:N_expand_ob}
\end{align}
with coefficients
\begin{align}
M_0 &= M(\Phi_0), &
M_1 &= M'(\Phi_0)\,\Phi_1, &
M_2 &= M'(\Phi_0)\,\Phi_2+\frac12 M''(\Phi_0)\,\Phi_1^2,
\label{eq:M_coeff_ob}\\
N_0 &= N(\Phi_0), &
N_1 &= N'(\Phi_0)\,\Phi_1, &
N_2 &= N'(\Phi_0)\,\Phi_2+\frac12 N''(\Phi_0)\,\Phi_1^2.
\label{eq:N_coeff_ob}
\end{align}
For the potential term we use
\begin{align}
W'(\Phi)
&=
W'(\Phi_0)
+\varepsilon\,W''(\Phi_0)\,\Phi_1
+\varepsilon^2\!\left(
W''(\Phi_0)\,\Phi_2+\frac12 W^{(3)}(\Phi_0)\,\Phi_1^2
\right)
+\mathcal{O}(\varepsilon^3).
\label{eq:Wp_expand_ob}
\end{align}
Finally, the geometric factors are expanded as
\begin{align}
(1+\varepsilon zH)^{-1} &= 1-\varepsilon zH+\varepsilon^2 z^2H^2+\mathcal{O}(\varepsilon^3),
\label{eq:geom1_ob}\\
(1+\varepsilon zH)^{-2} &= 1-2\varepsilon zH+\mathcal{O}(\varepsilon^2),
\qquad
(1+\varepsilon zH)^{-3} = 1-3\varepsilon zH+\mathcal{O}(\varepsilon^2).
\label{eq:geom23_ob}
\end{align}
\subsubsection{Leading inner profile and vanishing $\Upsilon_{-1}$}

\paragraph{Evolution equation at $\mathcal{O}(\varepsilon^{-3})$.}
From \eqref{eq:inner_evol_full_compact} the highest-order contribution comes from the normal part of
$\partial_z\big(M(\Phi)\partial_z(N(\Phi)\Upsilon)\big)$, yielding
\begin{equation}
0 = N_0\,\partial_z\!\Big(M_0\,\partial_z\big(N_0\,\Upsilon_{-1}\big)\Big).
\label{eq:evol_m3_ob}
\end{equation}
Integrating once in $z$, the quantity $M_0\,\partial_z(N_0\Upsilon_{-1})$ is constant across the layer. Matching the
normal flux with the outer regions and using $M(\pm1)=0$ implies this constant is zero, hence
$\partial_z(N_0\Upsilon_{-1})=0$. Matching with the outer expansion (where $\mu_{-1}=0$) gives
\begin{equation}
\Upsilon_{-1}\equiv 0.
\label{eq:Upsm1_zero_ob}
\end{equation}

\paragraph{Chemical potential equation at $\mathcal{O}(\varepsilon^{-1})$.}
Using \eqref{eq:inner_chem_full_compact}, the $\mathcal{O}(\varepsilon^{-1})$ balance is
\begin{equation}
0 = W'(\Phi_0)-\partial_{zz}\Phi_0.
\label{eq:AC_profile_ob}
\end{equation}
With matching $\Phi_0(\pm\infty)=\pm1$ and the gauge $\Phi_0(0)=0$, we obtain the classical heteroclinic profile
\begin{equation}
\Phi_0(z)=\sigma(z):=\tanh\!\Big(\frac{z}{\sqrt{2}}\Big),\qquad \sigma'=\frac{1}{\sqrt{2}}\,(1-\sigma^2).
\label{eq:sigma_ob}
\end{equation}
We record the constants (used below)
\begin{equation}
c_W:=\int_{-\infty}^{\infty}(\sigma')^2\,dz=\frac{2\sqrt{2}}{3},
\qquad
c_M:=\int_{-\infty}^{\infty}M(\sigma)\,dz=\int_{-\infty}^{\infty}(1-\sigma^2)^2\,dz=\frac{4\sqrt{2}}{3}.
\label{eq:cWcM_ob}
\end{equation}

\subsubsection{First correction and weighted chemical potential}

\paragraph{Evolution equation at $\mathcal{O}(\varepsilon^{-2})$.}
The next order gives
\begin{equation}
0 = N_0\,\partial_z\!\Big(M_0\,\partial_z\big(N_0\,\Upsilon_0\big)\Big),
\label{eq:evol_m2_ob}
\end{equation}
hence by the same flux-matching argument,
\begin{equation}
N(\sigma)\,\Upsilon_0 =: B_1(s,t),
\label{eq:weighted_mu_ob}
\end{equation}
i.e. the weighted chemical potential is constant across the layer at this order.

\paragraph{Chemical potential equation at $\mathcal{O}(1)$.}
Expanding \eqref{eq:inner_chem_full_compact} to $\mathcal{O}(1)$ and using $\Phi_0=\sigma$ yields
\begin{equation}
\Upsilon_0 = W''(\sigma)\,\Phi_1-\partial_{zz}\Phi_1 - H\,\sigma'.
\label{eq:chem_O1_ob}
\end{equation}
With \eqref{eq:weighted_mu_ob}, we may write $\Upsilon_0=B_1/N(\sigma)=B_1\,Q'(\sigma)$, and thus
\begin{equation}
\mathcal{L}\Phi_1
:=
\big(\partial_{zz}-W''(\sigma)\big)\Phi_1
=
-\frac{B_1}{N(\sigma)}-H\,\sigma'.
\label{eq:Phi1_eq_ob}
\end{equation}
Differentiating \eqref{eq:AC_profile_ob} gives $\mathcal{L}\sigma'=0$, hence
$\ker(\mathcal{L})=\mathrm{span}\{\sigma'\}$ and $\mathcal{L}$ is self-adjoint on $L^2(\mathbb{R})$.

\paragraph{Solvability condition and determination of $B_1$ (Fredholm alternative).}
Define the inner linearized operator
\begin{equation}
\mathcal{L}:=\partial_{zz}-W''(\sigma),
\label{eq:L_def_ob}
\end{equation}
so that \eqref{eq:Phi1_eq_ob} becomes
\begin{equation}
\mathcal{L}\Phi_1=-\frac{B_1}{N(\sigma)}-H\sigma'.
\label{eq:Phi1_eq_L_ob}
\end{equation}
The operator $\mathcal{L}$ is self-adjoint on $L^2(\mathbb{R})$ with domain $H^2(\mathbb{R})$.
Differentiating \eqref{eq:AC_profile_ob} yields $\mathcal{L}\sigma'=0$, hence
$\ker(\mathcal{L})=\mathrm{span}\{\sigma'\}$.
By the Fredholm alternative, a bounded solution $\Phi_1\in L^\infty(\mathbb{R})$ exists if and only if the right-hand side
of \eqref{eq:Phi1_eq_L_ob} is $L^2$-orthogonal to $\sigma'$:
\begin{equation}
\left\langle -\frac{B_1}{N(\sigma)}-H\sigma',\,\sigma'\right\rangle_{L^2(\mathbb{R})}=0,
\label{eq:Phi1_solv_ob}
\end{equation}
equivalently,
\begin{equation}
-B_1\int_{-\infty}^{\infty}\frac{\sigma'}{N(\sigma)}\,dz
\;-\;
H\int_{-\infty}^{\infty}(\sigma')^2\,dz
=0.
\label{eq:Phi1_solv_expanded_ob}
\end{equation}
The second integral is $c_W=\int_{-\infty}^{\infty}(\sigma')^2\,dz$.
For the first one we use $N(\phi)=1/Q'(\phi)$ and the substitution $\xi=\sigma(z)$ (so $d\xi=\sigma'(z)\,dz$):
\begin{equation}
c_N:=\int_{-\infty}^{\infty}\frac{\sigma'}{N(\sigma)}\,dz
=\int_{-\infty}^{\infty}Q'(\sigma)\sigma'\,dz
=\int_{-1}^{1}Q'(\xi)\,d\xi
=Q(1)-Q(-1)=2.
\label{eq:cN_ob}
\end{equation}
Thus \eqref{eq:Phi1_solv_expanded_ob} yields
\begin{equation}
B_1(s,t)=-\frac{c_W}{c_N}\,H(s,t)=-\frac{\sqrt{2}}{3}\,H(s,t).
\label{eq:B1_ob}
\end{equation}

\paragraph{Structure of $\Phi_1$.}
With \eqref{eq:B1_ob} and $N(\sigma)=1/Q'(\sigma)$, the forcing in \eqref{eq:Phi1_eq_L_ob} factorizes into a function of $z$
times the curvature $H(s,t)$. Consequently,
\begin{equation}
\Phi_1(z,s,t)=H(s,t)\,\Phi_1(z),
\label{eq:Phi1_factorization}
\end{equation}
where the reduced profile $\Phi_1(z)$ is the unique bounded solution (with the gauge $\Phi_1(0)=0$) of
\begin{equation}
\mathcal{L}\Phi_1
=\frac{\sqrt{2}}{3}\,Q'(\sigma(z))-\sigma'(z),
\qquad
\Phi_1(0)=0,\qquad \Phi_1 \text{ bounded}.
\label{eq:A_Phi1_ODE}
\end{equation}
For the NMN kernel $Q'(\phi)=\tfrac{3}{2}(1-\phi^2)$, one has $Q'(\sigma)=\tfrac{3\sqrt{2}}{2}\sigma'$ and thus the right-hand
side of \eqref{eq:A_Phi1_ODE} vanishes; hence $\Phi_1\equiv 0$.
A convenient integral (reduction-of-order / variation-of-constants) representation of the bounded solution of
\eqref{eq:A_Phi1_ODE} goes back to the inner-profile analysis in Bretin \emph{et al.}~\cite{bretin_2022}
(cf.\ their Lemma~4.1 and formula~(4.14)) and was later streamlined and reused in the ACH--IC framework of
Zhou \emph{et al.}~\cite{zhou_2025_achic}.
Explicit integral representations and reduced formulas for $\Phi_1$ (and the induced moment $\mathcal{J}_1[Q]$) are derived in
Appendix~\ref{app:B}, Section~\ref{sec:B_J1}.

\subsubsection{Interface velocity: $V_0=0$ and $V_1=\frac{4}{9}\Delta_\Gamma H$}

\paragraph{Vanishing leading velocity.}
At $\mathcal{O}(\varepsilon^{-1})$ in \eqref{eq:inner_evol_full_compact}, the time derivative first contributes through
$-(V_0/\varepsilon)\partial_z\Phi_0$. The remaining terms at this order involve $\partial_z(N_0\Upsilon_0)$ and
$\Upsilon_{-1}$; by \eqref{eq:Upsm1_zero_ob} and \eqref{eq:weighted_mu_ob} they vanish identically. Hence
$-V_0\,\sigma'=0$ and therefore
\begin{equation}
V_0=0.
\label{eq:V0_ob}
\end{equation}

\paragraph{Solvability at $\mathcal{O}(1)$.}
At $\mathcal{O}(1)$, with $V_0=0$ and $\Phi_0=\sigma(z)$, the evolution equation reduces to
\begin{equation}
-\,V_1\,\sigma'
=
N_0\,\partial_z\!\Big(M_0\,\partial_z(\,\cdots\,)\Big)
\;+\;
N_0\,\nabla_\Gamma\!\cdot\!\Big(M_0\,\nabla_\Gamma\big(N_0\Upsilon_0\big)\Big),
\label{eq:evol_O1_ob}
\end{equation}
where $(\cdots)$ collects the $\mathcal{O}(1)$ normal-flux corrections (its precise form is immaterial after integrating in $z$).
We multiply \eqref{eq:evol_O1_ob} by $1/N_0=Q'(\sigma)$ and integrate over $z\in\mathbb{R}$.
The first term on the right becomes a boundary term and vanishes because $M(\pm1)=0$. Using $N_0\Upsilon_0=B_1(s,t)$,
the tangential part yields
\begin{equation}
-\;V_1\int_{-\infty}^{\infty}Q'(\sigma)\sigma'\,dz
=
\int_{-\infty}^{\infty}\nabla_\Gamma\!\cdot\!\big(M(\sigma)\,\nabla_\Gamma B_1\big)\,dz
=
\nabla_\Gamma\!\cdot\!\left(\left[\int_{-\infty}^{\infty}M(\sigma)\,dz\right]\nabla_\Gamma B_1\right)
=
c_M\,\Delta_\Gamma B_1.
\label{eq:V1_solv_ob}
\end{equation}
By \eqref{eq:cN_ob}, $\int Q'(\sigma)\sigma'\,dz=c_N=2$, hence
\begin{equation}
V_1=-\frac{c_M}{c_N}\,\Delta_\Gamma B_1.
\label{eq:V1_intermediate_ob}
\end{equation}
Substituting \eqref{eq:B1_ob} and the constants \eqref{eq:cWcM_ob} gives
\begin{equation}
V_1=\frac{c_M\,c_W}{c_N^2}\,\Delta_\Gamma H=\frac{4}{9}\,\Delta_\Gamma H.
\label{eq:V1_ob}
\end{equation}
(For $d=2$, $\Delta_\Gamma=\partial_{ss}$.)

\paragraph{Resulting interfacial law.}
Recalling $V=V_0+\varepsilon V_1+\mathcal{O}(\varepsilon^2)$ and \eqref{eq:V0_ob}--\eqref{eq:V1_ob}, we obtain
\begin{equation}
V=\varepsilon\,\frac{4}{9}\,\Delta_\Gamma H+\mathcal{O}(\varepsilon^2),
\label{eq:interfacial_law_ob}
\end{equation}
i.e. surface diffusion dynamics on the physical time scale (the usual $\mathcal{O}(1)$ surface-diffusion speed is recovered by
a slow-time rescaling).

\section{Volume conservation and moment conditions}
\label{app:B}

This appendix derives the $\varepsilon$-expansion of the \emph{volume error} associated with the conserved quantity
$\int_\Omega Q(\phi_\varepsilon)\,\dOm$ in the CH--IC model. Our goal is twofold:
(i) obtain explicit and efficiently computable formulas for the leading error coefficients, and
(ii) identify kernel design conditions ensuring $\mathcal{O}(\varepsilon^3)$ (or better) approximation of the
sharp-phase volume.

The matched-asymptotic framework and the first inner correction problem are established in Appendix~\ref{app:A}.
Here we focus on:
\begin{enumerate}[label=(\roman*)]
  \item a bulk--interface decomposition of $\ErrV$ and the resulting order-by-order expansion,
  \item closed one-dimensional formulas for the \emph{geometric} moment $\mathcal{M}_1[Q]$ and the \emph{dynamic} moment
        $\mathcal{J}_1[Q]$,
  \item the \emph{moment-balance condition} $\mathcal{M}_1[Q]+\mathcal{J}_1[Q]=0$ cancelling the interfacial
        $\mathcal{O}(\varepsilon^2)$ error.
\end{enumerate}
The detailed reductions for $\Phi_1$ and $\mathcal{J}_1$ are a central analytical ingredient of the improved conservation
mechanism.

\subsection{Preliminaries and design class for $Q$}
\label{sec:B_prelim}

\subsubsection{Sharp vs.\ diffuse volume and the volume error}

Let $\Omega^\pm(t)$ be the sharp phases separated by $\Gamma(t)$. The sharp-phase volume is
\begin{equation}
|\Omega^+(t)|=\int_\Omega \Heav(d(x,t))\,\dOm,
\label{eq:B_Vsharp_def}
\end{equation}
where $d(x,t)$ is the signed distance to $\Gamma(t)$, positive in $\Omega^+(t)$ (Appendix~\ref{app:A}),
and $\Heav$ is the Heaviside step function.

In the CH--IC model we conserve the generalized mass
\begin{equation}
\mathcal{Q}[\phi_\varepsilon]
:=\int_\Omega Q(\phi_\varepsilon)\,\dOm,
\qquad Q(\pm1)=\pm1,\qquad Q\ \text{odd and strictly increasing.}
\label{eq:B_Qmass_def}
\end{equation}
Accordingly we define the conserved diffuse $Q$-volume proxy (cf.\ \eqref{eq:VQ_def} in the main text)
\begin{equation}
\VQ(t):=\frac12\int_\Omega \bigl(1+Q(\phi_\varepsilon)\bigr)\,\dOm,
\qquad \frac{d}{dt}\VQ(t)=0,
\label{eq:B_VQ_def}
\end{equation}
and the geometric-volume error (cf.\ \eqref{eq:EVdefinition})
\begin{equation}
\ErrV(t):=\VQ(t)-|\Omega^+(t)|
=\frac12\int_\Omega \bigl(1+Q(\phi_\varepsilon)\bigr)\,\dOm-\int_\Omega \Heav(d(x,t))\,\dOm.
\label{eq:B_EV_def}
\end{equation}

\subsubsection{Kernel family and normalization}

A convenient design class (used in the main text) is
\begin{equation}
Q'(\phi)=\frac{(1-\phi^2)^k\,S(\phi)}{B_k[S]},
\qquad
B_k[S]:=\int_0^1 (1-\xi^2)^k\,S(\xi)\,d\xi,
\label{eq:B_Q_general}
\end{equation}
where $k\ge 1$ is the endpoint degeneracy order and $S(\phi)>0$ is an even shaping function.
The normalization enforces
\begin{equation}
\int_{-1}^1 Q'(\phi)\,d\phi=2,
\qquad\text{equivalently}\qquad
Q(1)-Q(-1)=2,
\label{eq:B_Q_normalization}
\end{equation}
and with $Q$ odd implies $Q(1)=1$.

\begin{remark}[Normalization matters for shaped kernels]
When $S\not\equiv 1$, the constant $B_k[S]$ depends on the shaping parameters. In particular,
\[
Q(u)=\int_0^u Q'(t)\,dt
=\frac{\int_0^u (1-t^2)^kS(t)\,dt}{\int_0^1 (1-t^2)^kS(t)\,dt},
\qquad u\in[-1,1],
\]
and it is the ratio (not only the numerator) that guarantees $Q(1)=1$.
\end{remark}

\subsection{Bulk--interface decomposition and the $\varepsilon$-expansion}
\label{sec:B_decomp}

Let $d(x,t)$ be the signed distance to $\Gamma(t)$, positive in $\Omega^+(t)$.
We use the standard decomposition with $\delta=C\varepsilon|\ln\varepsilon|$,
\[
\Omega=\Omega^+_\delta(t)\cup \Gamma_\delta(t)\cup \Omega^-_\delta(t),
\qquad
\Gamma_\delta(t):=\{x:|d(x,t)|\le \delta\}.
\]
In the interfacial tube we introduce the stretched coordinate $z=d/\varepsilon$ and the surface variable $s\in\Gamma(t)$,
so that $x=X(s,t)+\varepsilon z\,\nu(s,t)$. The Jacobian has the curvature expansion
\begin{equation}
\dOm=\varepsilon\,\mathcal{J}(\varepsilon z,s)\,dz\,dA(s),
\qquad
\mathcal{J}(\varepsilon z,s)
=1+\varepsilon z H(s)+\varepsilon^2 z^2 G(s)+\mathcal{O}(\varepsilon^3),
\label{eq:B_Jacobian_expand}
\end{equation}
where $H=\sum_{i=1}^{d-1}\kappa_i$ and $G=\sum_{i<j}\kappa_i\kappa_j$ (for $d=3$, $G$ is the Gaussian curvature).

In the inner region, Appendix~\ref{app:A} gives the profile expansion
\begin{equation}
\phi_\varepsilon(x,t)=\Phi(z,s,t)
=\sigma(z)+\varepsilon\,H(s,t)\,\Phi_1(z)+\mathcal{O}(\varepsilon^2),
\qquad
\sigma(z)=\tanh(z/\sqrt2),
\label{eq:B_inner_phi_exp}
\end{equation}
with the factorization $\Phi_1(z,s,t)=H(s,t)\Phi_1(z)$ and $\Phi_1(z)$ solving the reduced ODE
\eqref{eq:A_Phi1_ODE}.

\subsubsection{Leading cancellation}
Insert \eqref{eq:B_inner_phi_exp} into \eqref{eq:B_EV_def} restricted to $\Gamma_\delta(t)$ and expand
\[
Q(\Phi)=Q(\sigma)+\varepsilon\,H\,\Phi_1\,Q'(\sigma)+\mathcal{O}(\varepsilon^2),
\qquad
\mathcal{J}=1+\varepsilon zH+\mathcal{O}(\varepsilon^2).
\]
The interfacial contribution becomes
\[
\ErrV^{\mathrm{int}}
=
\frac12\int_{\Gamma}\int_{-\infty}^{\infty}
\Big(Q(\sigma(z))-(2\Heav(z)-1)\Big)\,\varepsilon\,dz\,dA
+\mathcal{O}(\varepsilon^2),
\]
which vanishes at order $\varepsilon$ by odd/even symmetry (since $Q$ and $\sigma$ are odd, while $\Heav(z)-\tfrac12$
is odd). Thus the leading interfacial error is $\mathcal{O}(\varepsilon^2)$.

\subsection{The $\mathcal{O}(\varepsilon^2)$ interfacial error and two moments}
\label{sec:B_E2}

Carrying the above expansion one order further yields
\begin{equation}
\ErrV^{(2)}
=\frac{\varepsilon^2}{2}\int_{\Gamma(t)} H(s)\,\Bigl(\mathcal{M}_1[Q]+\mathcal{J}_1[Q]\Bigr)\,dA(s),
\label{eq:B_E2_simplified}
\end{equation}
where the \emph{geometric moment} and the \emph{dynamic moment} are defined by
\begin{align}
\mathcal{M}_1[Q]
&:=\int_{-\infty}^{\infty} z\,\Bigl(1+Q(\sigma(z))-2\Heav(z)\Bigr)\,dz,
\label{eq:B_M1_def}\\[2mm]
\mathcal{J}_1[Q]
&:=\int_{-\infty}^{\infty}\Phi_1(z)\,Q'(\sigma(z))\,dz.
\label{eq:B_J1_def}
\end{align}
Both moments depend only on the kernel $Q$ (through $Q$ and $Q'$), while geometry enters \eqref{eq:B_E2_simplified} only through
the factor $H(s)$.

\begin{remark}[Interpretation]
$\mathcal{M}_1$ comes from the Jacobian correction $\varepsilon zH$ (a purely geometric mismatch between diffuse and sharp tubes),
while $\mathcal{J}_1$ comes from the curvature-induced profile correction $\varepsilon H\Phi_1$.
The key design freedom is to tune $Q$ so that these two effects cancel.
\end{remark}

\subsection{Efficient formula and sign of the geometric moment $\mathcal{M}_1[Q]$}
\label{sec:B_M1}

\begin{theorem}[Single-integral formula for $\mathcal{M}_1$]
\label{thm:B_M1_arctanh2}
For any odd, increasing $Q$ with $Q(1)=1$,
\begin{equation}
\mathcal{M}_1[Q] = -2\int_0^1 Q'(u)\,\operatorname{arctanh}^2(u)\,du.
\label{eq:B_M1_arctanh2}
\end{equation}
In particular, for \eqref{eq:B_Q_general},
\[
\mathcal{M}_1[Q]
=
-\frac{2}{B_k[S]}\int_0^1 (1-u^2)^k S(u)\,\operatorname{arctanh}^2(u)\,du.
\]
\end{theorem}

\begin{proof}
Split \eqref{eq:B_M1_def} at $z=0$ and use $\Heav(z)=1$ for $z>0$, $\Heav(z)=0$ for $z<0$:
\[
\mathcal{M}_1
=
\int_{-\infty}^{0} z\,(1+Q(\sigma(z)))\,dz
+\int_{0}^{\infty} z\,(Q(\sigma(z))-1)\,dz.
\]
With $z=-\eta$ in the first term and oddness of $Q$ and $\sigma$,
\[
\int_{-\infty}^{0} z\,(1+Q(\sigma(z)))\,dz
=
-\int_{0}^{\infty}\eta\,(1-Q(\sigma(\eta)))\,d\eta,
\qquad
\int_{0}^{\infty} z\,(Q(\sigma(z))-1)\,dz
=
-\int_{0}^{\infty}z\,(1-Q(\sigma(z)))\,dz,
\]
hence $\mathcal{M}_1=-2\int_{0}^{\infty} z\,(1-Q(\sigma(z)))\,dz$.
Now substitute $u=\sigma(z)=\tanh(z/\sqrt2)$, so $z=\sqrt2\,\operatorname{arctanh}(u)$ and
$dz=\frac{\sqrt2}{1-u^2}\,du$, which gives
\[
\mathcal{M}_1
=
-4\int_{0}^{1}\frac{(1-Q(u))\,\operatorname{arctanh}(u)}{1-u^2}\,du.
\]
Write $1-Q(u)=\int_u^1 Q'(\zeta)\,d\zeta$ and apply Fubini:
\[
\mathcal{M}_1
=
-4\int_0^1 Q'(\zeta)\left(\int_0^\zeta \frac{\operatorname{arctanh}(u)}{1-u^2}\,du\right)d\zeta.
\]
Since $\frac{d}{du}\big(\tfrac12\operatorname{arctanh}^2(u)\big)=\frac{\operatorname{arctanh}(u)}{1-u^2}$,
the inner integral is $\tfrac12\operatorname{arctanh}^2(\zeta)$, yielding \eqref{eq:B_M1_arctanh2}.
\end{proof}

\begin{lemma}[Strict negativity of $\mathcal{M}_1$]
\label{lem:M1negative_appendix}
If $Q$ is strictly increasing on $(-1,1)$, then $\mathcal{M}_1[Q]<0$.
\end{lemma}

\begin{proof}
In \eqref{eq:B_M1_arctanh2}, the integrand $Q'(u)\operatorname{arctanh}^2(u)$ is strictly positive for $u\in(0,1)$.
\end{proof}

\begin{remark}[Closed forms for $S\equiv 1$ and why shaped kernels are different]
\label{rem:M1_S1_closed}
For $S\equiv 1$ define
\[
I_k:=\int_0^1 (1-u^2)^k\,\operatorname{arctanh}^2(u)\,du.
\]
With the substitution $x=\operatorname{arctanh}(u)$ (so $u=\tanh x$, $du=\operatorname{sech}^2(x)\,dx$,
and $1-u^2=\operatorname{sech}^2(x)$), one obtains
\[
I_k=\int_0^\infty x^2\,\operatorname{sech}^{2k+2}(x)\,dx.
\]
A classical evaluation (e.g.\ via beta/gamma identities together with an integration-by-parts recurrence) yields, for
all integers $k\ge 0$,
\begin{equation}
\label{eq:B_Ik_sech_closed}
\int_0^\infty x^2\,\operatorname{sech}^{2k+2}(x)\,dx
=
\frac{2^{\,2k-1}(k!)^2}{(2k+1)!}
\left(\frac{\pi^2}{6}-\sum_{j=1}^{k}\frac{1}{j^2}\right),
\end{equation}
(with the convention $\sum_{j=1}^0(\cdot)=0$, so the $k=0$ case gives $\pi^2/12$).

For the unshaped normalized kernel
\[
Q_k'(u)=\frac{(1-u^2)^k}{\int_0^1(1-\xi^2)^k\,d\xi},
\qquad
\int_0^1(1-\xi^2)^k\,d\xi=\frac{2^{\,2k}(k!)^2}{(2k+1)!},
\]
the geometric moment \eqref{eq:B_M1_arctanh2} simplifies to the closed form
\[
\mathcal{M}_1[Q_k]
=
\sum_{j=1}^{k}\frac{1}{j^2}-\frac{\pi^2}{6}.
\]
In particular,
\[
\mathcal{M}_1[Q_1]=1-\frac{\pi^2}{6},\qquad
\mathcal{M}_1[Q_2]=\frac{5}{4}-\frac{\pi^2}{6},\qquad
\mathcal{M}_1[Q_3]=\frac{49}{36}-\frac{\pi^2}{6}.
\]

For general shaping $S(u)$, the same substitution produces weighted integrals of the form
\[\int_0^\infty x^2\,\operatorname{sech}^{2k+2}(x)\,S(\tanh x)\,dx\,,\]
for which no comparable closed form is available in
general; nevertheless the single-quadrature representation \eqref{eq:B_M1_arctanh2} remains robust and inexpensive.
\end{remark}

\subsection{Reduced formulas for $\Phi_1$ and the dynamic moment $\mathcal{J}_1[Q]$}
\label{sec:B_J1}

We now derive closed one-dimensional formulas for the dynamic moment \eqref{eq:B_J1_def}. The main analytical input is the
first-order inner problem from Appendix~\ref{app:A}. With
$\sigma(z)=\tanh(z/\sqrt2)$ and $\mathcal{L}=\partial_{zz}-W''(\sigma)$, the reduced profile $\Phi_1(z)$ solves
\begin{equation}
\mathcal{L}\Phi_1
=\frac{\sqrt{2}}{3}Q'(\sigma(z))-\sigma'(z),
\qquad
\Phi_1(0)=0,\qquad \Phi_1 \text{ bounded}.
\label{eq:B_Phi1_ODE}
\end{equation}
Equivalently, $\mathcal{L}\Phi_1=A_1(z)\sigma'(z)$ with
\[
A_1(z)=\frac{\sqrt{2}Q'(\sigma(z))}{3\sigma'(z)}-1.
\]

\subsubsection{Integral representation (Bretin; Zhou)}

A reduction-of-order / variation-of-constants representation for the bounded solution of \eqref{eq:B_Phi1_ODE} appears already
in Bretin \emph{et al.}~\cite{bretin_2022} (see Lemma~4.1 and formula~(4.14)) and is revisited and streamlined in the ACH--IC
setting by Zhou \emph{et al.}~\cite{zhou_2025_achic}.

\begin{theorem}[Variation-of-constants formula]
\label{thm:Zhou_formula}
If $\Phi_1$ solves \eqref{eq:B_Phi1_ODE} and is bounded with $\Phi_1(0)=0$, then
\begin{equation}
\Phi_1(z)
=
\sigma'(z)\int_0^{z}\frac{1}{\sigma'(\eta)^2}
\left(\int_0^{\eta} A_1(\xi)\,\sigma'(\xi)^2\,d\xi\right)\,d\eta.
\label{eq:B_Zhou_formula}
\end{equation}
\end{theorem}

\begin{proof}
This is the standard reduction-of-order construction for the self-adjoint operator $\mathcal{L}$, using that
$\sigma'\in\ker(\mathcal{L})$ and fixing the homogeneous component by the gauge $\Phi_1(0)=0$.
\end{proof}

\subsubsection{Single-quadrature reduction for $\Phi_1$}

The double integral \eqref{eq:B_Zhou_formula} can be reduced to a single integral in the phase variable $u=\sigma(z)$.

\begin{theorem}[Reduced formula for $\Phi_1$]
\label{thm:Phi1_reduced}
Define the NMN reference kernel
\[
Q_1(u)=\frac{3}{2}\left(u-\frac{u^3}{3}\right),
\qquad
Q_1'(u)=\frac{3}{2}(1-u^2).
\]
Then the bounded solution of \eqref{eq:B_Phi1_ODE} is
\begin{equation}
\Phi_1(z)
=
\frac{4}{3}\,\sigma'(z)\int_0^{\sigma(z)}\frac{Q(u)-Q_1(u)}{(1-u^2)^3}\,du.
\label{eq:B_Phi1_reduced}
\end{equation}
\end{theorem}

\begin{proof}
Starting from \eqref{eq:B_Zhou_formula}, substitute $u=\sigma(\xi)$ and use
$\sigma'=\frac{1-\sigma^2}{\sqrt2}$ and $d\xi=\frac{\sqrt2}{1-u^2}\,du$.
A direct computation gives
\[
\int_0^\eta A_1(\xi)\sigma'(\xi)^2\,d\xi
=
\frac{\sqrt2}{3}\bigl(Q(\sigma(\eta))-Q_1(\sigma(\eta))\bigr).
\]
Substituting this expression back into \eqref{eq:B_Zhou_formula} and changing variables once more from $\eta$ to
$u=\sigma(\eta)$ yields \eqref{eq:B_Phi1_reduced}.
\end{proof}

\subsubsection{A closed one-dimensional formula for $\mathcal{J}_1[Q]$}

\begin{theorem}[Single-integral formula for the dynamic moment]
\label{thm:J1_single}
Assume $Q(\pm1)=\pm1$ and that $Q$ is odd and strictly increasing.
Then
\begin{equation}
\mathcal{J}_1[Q]
=
\frac{8}{3}\int_0^1
\frac{\bigl(Q(u)-Q_1(u)\bigr)\,\bigl(1-Q(u)\bigr)}{(1-u^2)^3}\,du.
\label{eq:B_J1_single}
\end{equation}
\end{theorem}

\begin{proof}
Insert \eqref{eq:B_Phi1_reduced} into \eqref{eq:B_J1_def} and use $u=\sigma(z)$:
\[
\mathcal{J}_1[Q]
=
2\int_0^\infty Q'(\sigma(z))\Phi_1(z)\,dz
=
\frac{8}{3}\int_0^1 Q'(u)\left(\int_0^u \frac{Q(v)-Q_1(v)}{(1-v^2)^3}\,dv\right)\,du.
\]
Interchanging the order of integration (Fubini) yields
\[
\mathcal{J}_1[Q]
=
\frac{8}{3}\int_0^1
\frac{Q(v)-Q_1(v)}{(1-v^2)^3}
\left(\int_v^1 Q'(u)\,du\right)\,dv
=
\frac{8}{3}\int_0^1
\frac{\bigl(Q(v)-Q_1(v)\bigr)\,\bigl(1-Q(v)\bigr)}{(1-v^2)^3}\,dv,
\]
since $\int_v^1 Q'(u)\,du=Q(1)-Q(v)=1-Q(v)$.
\end{proof}

\begin{remark}[Integrability and the role of $k\ge 2$]
Near $u=1$, one has $1-Q(u)\sim C(1-u^2)^{k+1}$ for the family \eqref{eq:B_Q_general}.
The integrand in \eqref{eq:B_J1_single} behaves like $(1-u^2)^{k-1}$, which is integrable if and only if $k\ge 2$.
Thus $k\ge 2$ is the natural threshold for a finite $\mathcal{J}_1$ and for bulk error suppression of order
$\varepsilon^{k+1}\ge\varepsilon^3$.
\end{remark}

\subsection{Moment balance and $\mathcal{O}(\varepsilon^3)$ volume accuracy}
\label{sec:B_balance}

Define the combined interfacial coefficient
\begin{equation}
\mathcal{C}_1[Q]:=\mathcal{M}_1[Q]+\mathcal{J}_1[Q].
\label{eq:B_C1_def}
\end{equation}
By \eqref{eq:B_E2_simplified}, the leading interfacial volume error is proportional to $\mathcal{C}_1[Q]$.

\begin{theorem}[Moment-balance condition]
\label{thm:moment_balance}
Suppose $k\ge 2$ (so that the bulk contribution is $\mathcal{O}(\varepsilon^{k+1})=\mathcal{O}(\varepsilon^3)$ and
$\mathcal{J}_1$ is finite). If
\begin{equation}
\mathcal{C}_1[Q]=\mathcal{M}_1[Q]+\mathcal{J}_1[Q]=0,
\label{eq:B_balance_condition}
\end{equation}
then the total volume error satisfies $\ErrV(t)=\mathcal{O}(\varepsilon^3)$.
\end{theorem}

\begin{remark}[Sign mechanism]
By Lemma~\ref{lem:M1negative_appendix}, admissible kernels have $\mathcal{M}_1[Q]<0$.
For many kernels with $k\ge 2$ one finds $\mathcal{J}_1[Q]>0$, so cancellation through
\eqref{eq:B_balance_condition} is feasible within the strictly positive class $S>0$.
\end{remark}

\begin{remark}[Example: exponential shaping]
For the one-parameter family $S(u)=\exp(-b_2 u^2)$ with $k=2$, the balance equation \eqref{eq:B_balance_condition} admits a root
(under our normalization convention) at approximately $b_2\approx 6.95$. This value is included only as an illustration; the
main text reports the numerical optimization results.
\end{remark}

\subsection{Bulk contribution and degeneracy order}
\label{sec:B_bulk}

The interfacial analysis above determines the leading $\mathcal{O}(\varepsilon^2)$ term. The bulk contribution is governed by
the endpoint degeneracy of $Q$.

\begin{theorem}[Bulk error scaling]
\label{thm:B_bulk_error}
If $Q'(\pm1)=Q''(\pm1)=\cdots=Q^{(k)}(\pm1)=0$ (degeneracy order $k$), then the bulk contribution satisfies
\[
\ErrV^{\mathrm{bulk}}(t)=\mathcal{O}(\varepsilon^{k+1}).
\]
\end{theorem}
\noindent
For $k=1$ (NMN-type kernels) the bulk and interfacial errors both enter at $\mathcal{O}(\varepsilon^2)$.
For $k\ge 2$ the bulk contribution is at least $\mathcal{O}(\varepsilon^3)$, and the leading error is purely interfacial,
controlled by $\mathcal{C}_1[Q]$.

\section{Discrete conservation and energy stability}\label{app:C}

This appendix records two structural properties of the proposed discretization:
(i) discrete conservation of the designed invariant $\int_\Omega Q(\phi)\,\dOm$; and
(ii) unconditional dissipation of the diffuse-interface energy for the \emph{variational} CH--IC subclass (i.e.\ $R\equiv 0$).
The arguments follow standard convex-splitting proofs for Cahn--Hilliard schemes, but we emphasize the discrete chain rule
needed for an exact $Q^{n+1}-Q^n$ update.
A constant mobility prefactor can be absorbed into the time scale, so we set it to~$1$ without loss of generality.

\subsection{Discrete $Q$-conservation}\label{app:C:conservation}

Consider the fully discrete finite-volume form of \eqref{eq:chic_mixed_Q} with implicit Euler time stepping,
\begin{equation}\label{eq:disc_Q_cons}
\frac{Q(\phi^{n+1})-Q(\phi^n)}{\delta t}
=
\nabla_h\cdot\Big(M(\phi^\star)\,\nabla_h \psi^{n+1}\Big),
\end{equation}
where $\nabla_h\cdot(\cdot)$ denotes the conservative FV divergence and $\phi^\star$ is any chosen evaluation
state (e.g.\ $\phi^n$ or a Picard iterate).
Summing \eqref{eq:disc_Q_cons} over all control volumes and using no-flux boundary conditions yields
\[
\sum_{i=1}^{N} Q(\phi_i^{n+1})\,|\Omega_i|
=
\sum_{i=1}^{N} Q(\phi_i^{n})\,|\Omega_i|,
\]
up to linear-solver tolerance. Hence the discrete invariant is preserved once the nonlinear solve converges.

\subsection{Unconditional energy dissipation for the variational CH--IC subclass}\label{app:C:energy}

We restrict to the variational CH--IC case $R(\phi,\nabla\phi)\equiv 0$ in \eqref{eq:chic_mixed_psi}.
In the nondimensionalization used in the main text we set the surface-tension prefactor to $1$, so the standard
diffuse-interface energy reads
\begin{equation}\label{eq:energy_appC}
E_\varepsilon[\phi]
=
\int_\Omega \Big(\frac{1}{\varepsilon}W(\phi)+\frac{\varepsilon}{2}|\nabla \phi|^2\Big)\,\dOm.
\end{equation}
Let $W=W_c-W_e$ be a convex splitting with $W_c$ convex and $W_e$ concave \cite{eyre_1998}.

\paragraph{Discrete chain rule.}
Define the pointwise difference quotient
\begin{equation}\label{eq:Qhat_def}
\widehat{Q}'(\phi^{n+1},\phi^n)
:=
\begin{cases}
\dfrac{Q(\phi^{n+1})-Q(\phi^n)}{\phi^{n+1}-\phi^n}, & \phi^{n+1}\neq \phi^n,\\[8pt]
Q'(\phi^n), & \phi^{n+1}=\phi^n,
\end{cases}
\end{equation}
which is nonnegative whenever $Q$ is monotone increasing. By construction,
\begin{equation}\label{eq:disc_chain_rule}
Q(\phi^{n+1})-Q(\phi^n) = \widehat{Q}'(\phi^{n+1},\phi^n)\,(\phi^{n+1}-\phi^n)
\qquad\text{pointwise in }\Omega.
\end{equation}

\paragraph{Energy-stable fully discrete scheme.}
Consider the mixed scheme (the nonlinear ``ideal'' counterpart of the Picard-linearized implementation)
\begin{subequations}\label{eq:scheme_appC}
\begin{align}
\frac{Q(\phi^{n+1})-Q(\phi^n)}{\delta t}
&=
\nabla\cdot\!\Big(M(\phi^{n})\,\nabla \psi^{n+1}\Big),
\label{eq:scheme_appC_Q}
\\
\widehat{Q}'(\phi^{n+1},\phi^n)\,\psi^{n+1}
&=
\frac{1}{\varepsilon}\Big(W_c'(\phi^{n+1})-W_e'(\phi^n)\Big)\;-\;\varepsilon\,\Delta \phi^{n+1}.
\label{eq:scheme_appC_psi}
\end{align}
\end{subequations}

\begin{theorem}[Unconditional energy dissipation]\label{thm:energy_stability_appC}
Let $W=W_c-W_e$ with $W_c$ convex and $W_e$ concave, and assume no-flux boundary conditions.
Then \eqref{eq:scheme_appC} satisfies the discrete energy inequality
\begin{equation}\label{eq:energy_ineq_appC}
E_\varepsilon[\phi^{n+1}]-E_\varepsilon[\phi^{n}]
\;\le\;
-\delta t\int_\Omega M(\phi^n)\,|\nabla \psi^{n+1}|^2\,\dOm
\;-\;\frac{\varepsilon}{2}\int_\Omega |\nabla(\phi^{n+1}-\phi^n)|^2\,\dOm,
\end{equation}
for any $\delta t>0$. In particular, $E_\varepsilon[\phi^{n+1}]\le E_\varepsilon[\phi^n]$.
\end{theorem}

\begin{proof}
\emph{Step 1: The $Q$-equation.}
Multiply \eqref{eq:scheme_appC_Q} by $\psi^{n+1}$ and integrate over $\Omega$.
Using integration by parts and the no-flux condition gives
\begin{equation}\label{eq:step1_appC}
\int_\Omega \psi^{n+1}\big(Q(\phi^{n+1})-Q(\phi^n)\big)\,\dOm
=
-\delta t\int_\Omega M(\phi^n)\,|\nabla \psi^{n+1}|^2\,\dOm.
\end{equation}

\emph{Step 2: Discrete chain rule.}
Multiply \eqref{eq:scheme_appC_psi} by $(\phi^{n+1}-\phi^n)$ and integrate.
By \eqref{eq:disc_chain_rule}, the left-hand side is exactly
\[
\int_\Omega \widehat{Q}'(\phi^{n+1},\phi^n)\,\psi^{n+1}(\phi^{n+1}-\phi^n)\,\dOm
=
\int_\Omega \psi^{n+1}\big(Q(\phi^{n+1})-Q(\phi^n)\big)\,\dOm.
\]
Hence,
\begin{align}
\int_\Omega \psi^{n+1}\big(Q(\phi^{n+1})-Q(\phi^n)\big)\,\dOm
&=
\frac{1}{\varepsilon}\int_\Omega\big(W_c'(\phi^{n+1})-W_e'(\phi^n)\big)(\phi^{n+1}-\phi^n)\,\dOm
\nonumber\\
&\qquad
-\varepsilon\int_\Omega \Delta \phi^{n+1}(\phi^{n+1}-\phi^n)\,\dOm.
\label{eq:step2_appC}
\end{align}

\emph{Step 3: Convexity/concavity inequalities.}
By convexity of $W_c$ and concavity of $W_e$,
\[
W_c(\phi^{n+1})-W_c(\phi^n)\le W_c'(\phi^{n+1})(\phi^{n+1}-\phi^n),
\qquad
W_e(\phi^{n+1})-W_e(\phi^n)\ge W_e'(\phi^n)(\phi^{n+1}-\phi^n).
\]
Subtracting yields
\begin{equation}\label{eq:Wsplitting_ineq_appC}
W(\phi^{n+1})-W(\phi^n)
\le
\big(W_c'(\phi^{n+1})-W_e'(\phi^n)\big)(\phi^{n+1}-\phi^n).
\end{equation}

\emph{Step 4: Gradient term.}
Integrating by parts in the Laplacian contribution gives
\[
-\int_\Omega \Delta \phi^{n+1}(\phi^{n+1}-\phi^n)\,\dOm
=
\int_\Omega \nabla \phi^{n+1}\cdot\nabla(\phi^{n+1}-\phi^n)\,\dOm
=
\frac12\Big(\|\nabla\phi^{n+1}\|_2^2-\|\nabla\phi^n\|_2^2+\|\nabla(\phi^{n+1}-\phi^n)\|_2^2\Big).
\]
Combining \eqref{eq:step1_appC}--\eqref{eq:step2_appC} with \eqref{eq:Wsplitting_ineq_appC} and the gradient identity
yields \eqref{eq:energy_ineq_appC}.
\end{proof}

\begin{remark}[Relation to the implemented Picard/block solve and to earlier general proofs]
The proof above is for the fully discrete nonlinear scheme \eqref{eq:scheme_appC}.
In practice we solve the corresponding system by Picard linearization and a preconditioned, block-coupled GMRES (or BiCGStab) method; once the nonlinear
loop converges, the discrete conservation identity \eqref{eq:disc_Q_cons} is satisfied to solver tolerance, and the energy
decay \eqref{eq:energy_ineq_appC} is observed in practice.
More general energy-stability results for weighted-metric variants (including additional variational terms in
\eqref{eq:chic_mixed_psi} and generalized prefactors) are given in \cite{musil_furst_2025_enhanced}; the same discrete
difference-quotient idea \eqref{eq:Qhat_def} provides the exact chain rule needed to close the dissipation estimate.
\end{remark}

\bibliographystyle{unsrt} 
\bibliography{phase_field_v2}
\end{document}